\documentclass[final, nomarks]{dmtcs-episciences}

\usepackage[utf8]{inputenc}
\usepackage{subfigure}

\usepackage[round]{natbib}
\usepackage{amsmath}
\newtheorem{lemma}{Lemma}
\newtheorem{proposition}{Proposition}
\newtheorem{remark}{Remark}
\newtheorem{cor}{Corollary}
\newtheorem{theorem}{Theorem}

\def\exp{\textup{exp}}
\def\Poisson{\textup{Poisson}}

\def\ln{\textup{ln}}

\author{Silvia Businger}
\title[Occupancy scheme in a random environment]{Asymptotics of the occupancy scheme \\ in a random environment and its  applications to tries}
\affiliation{
  Universität Zürich, Switzerland}
\received{2016-9-19}

\accepted{2017-7-26}

\begin{document}
\publicationdetails{19}{2017}{1}{22}{2037}
\maketitle
\begin{abstract}
 Consider $ m $ copies of  an irreducible, aperiodic Markov chain $ Y $  taking values in a finite state space. The asymptotics as $ m $ tends to infinity, of  the first time from which on the trajectories of the $ m $ copies differ, have been studied by Szpankowski (1991) in the setting of tries. We use a different approach and model the $ m $ trajectories by a variant of the occupancy scheme, where we consider a nested sequence of boxes. This approach will enable  us to extend the result to  the case when the transition probabilities are random. We moreover use the same techniques to study the asymptotics as $ m $ tends to infinity of the time up to which we have observed all the possible trajectories of $ Y $ in random and nonrandom scenery. 
\end{abstract}

\section{Introduction}

 Let $Y= (Y_k)_{k \in \mathbb{N}} $ be an irreducible, aperiodic Markov chain  taking values in a finite state space, say $ \Sigma= \lbrace 1,...,K \rbrace $, denote its transition probabilities by $ (p_{ij}, i,j \in \Sigma) $, and consider $ m $ independent copies of $ Y $.  We define $ H_m \in \mathbb{N}$,  the first time from which on the trajectories of the $ m $ copies differ and $ G_m \in \mathbb{N} $, the  maximal time up to which we have observed all the possible trajectories of $ Y $. A natural question which has generated a lot of interest in the literature concerns the asymptotic behavior of $ H_m $ and $ G_m $ as $ m $ tends to infinity. 
 
Typically in  information theory one considers the Markov source model (see for example \cite{Ash}), that is an infinite data string is modeled by the trajectory of a Markov chain taking values in a finite alphabet $ \Sigma $. We call the first $ n $ letters of the data string  a word of length $ n $. If we consider $ m $ independent data strings modeled by a Markov source, then $ H_{m} $  is the minimal length $ n $ such that the $ m $ words with length $ n $ are all distinct. In the same spirit $ G_{m} $ describes up to which length we observe all the possible words. More generally we will consider $ H_{m,j} $ the minimal length such that out of the $ m $ words with length $ n $ at most $ j-1 $ words are non-distinct, and $ G_{m,j} $ describing up to which length we observe all the possible words at least $ j $ times.

The length $ H_m $ can be described in terms of a class of combinatorial trees, the so called $ K $-ary tries that can be  constructed with an iterative procedure. Consider $ m $  words with letters in the alphabet $ \Sigma $. We start with the root of a tree. We then look  at the first letter of  the $ m $ words. For each different letter we add  a node to the root, in increasing order. If for a letter there is only one word starting with this letter, then the corresponding node will be a leaf of the tree and we store the word in that leaf. We then look at the second letters of the words not yet being stored in a leaf. If  there are at least two words starting with the letter $ i \in \Sigma$, we then look at the second letters of all the words starting with $ i $. For each different letter we add a node to the node $ i $ in the same way as in the first step.  We then repeat this procedure until all the words are stored in a leaf. For more details see for example \cite{drmota}. When the data strings have been generated by a Markov source then the height of this $ K $-ary trie is equal to  $ H_m $. 

The height of tries have been a subject in research for many years, see  \cite{Flajolet} and  \cite{Devroye}. The limit law of the height has finally been derived by \cite{Pittel}, see also \cite{Pittel2}.  The Markov source model, was studied by   \cite{S}. He established that
\[   H_{m}  \sim \frac{2}{-\ln(\rho(2))}\ln(m),  \quad a.s.,\]
as $ m $ tends to infinity, where $ \rho(2) $ denotes the maximum modulus eigenvalue of the matrix $ A(2)= (p_{ik}^{2})_{i,k \in \Sigma} $, by analyzing the longest common prefix of each possible pair of two data strings. 

This result can be extended  to the case when each external node is allowed to store up to $ j-1 $ data strings. The tree defined by this modification is called a $ (j-1)$-trie. Szpankowski (see \cite{S} and \cite{S2}) studied the asymptotics of the height of a $ (j-1) $-trie, he established that
\[H_{m,j} \sim \frac{j}{-\ln(\rho(j))}\ln(m),  \quad a.s.,\]
as $ m $ tends to infinity, where $ \rho(j) $ denotes the maximum modulus eigenvalue of the matrix given by $ A(j)= (p_{ik}^{j})_{i,k \in \Sigma} $.

Our purpose in this work is to investigate a similar problem in random environment. This is a  natural model if for instance one wants to take into account transmission errors in the setting of information theory.
 
 In order to explain how, in our model, the random environment acts on the trajectories of chains, let us first recall a simple construction of i.i.d. copies of the Markov chain $Y$ with transition matrix $(p_{ik})_{i,k \in \Sigma}$ using the balls-in-bins setting (see, e.g.  \cite{Devroye2}).  For a general introduction to balls-in-bins  in deterministic environment and the classical occupancy scheme with finitely many boxes see for example  \cite{Kolchin}  or  \cite{Kotz}. There is a broad literature on  the occupancy scheme in deterministic environment and comparatively few investigations of the occupancy scheme in random environment. An infinite occupancy scheme  in a random environment called the Bernoulli sieve has been introduced in \cite{Gnedin}  and then studied in depth, a survey can be found for example in \cite{Gnedin2}. A scheme of a similar type can also be found in \cite{Robert}.

The set of words with length $n$ in the alphabet $\Sigma$ is $\Sigma^{n}$, with the convention that for $n=0$, $\Sigma^{0}=\{ \emptyset \}$ is the empty word. For the sake of simplicity,
we assume that the initial state of the Markov chain $Y$ is always $1$, and construct a nested family of boxes (or bins) indexed by the regular $ K $-ary tree $\mathcal{U}= \bigcup_{n \in\mathbb{Z}_+}\Sigma^{n}$ as follows. At generation $ 0$, there is a single box $b_{\emptyset}$ with unit size and type $1$. At the first generation, we divide the box $b_{\emptyset}$ into $(b_i: i\in \Sigma)$ where $b_i$ has type $i$ and size $\vert b_i \vert=p_{1i}$. We iterate for the next generations in an obvious way. For each word ${u}=(i_1, \ldots, i_n)\in \Sigma^n$,
the box $b_{{u}}$ has type $i_n$ and is split at the next generation  into sub-boxes $(b_{{u} k}: k\in \Sigma)$, where $b_{{u} k}$ has type $k$ and size $|b_{{u} k}|=|b_{{u}}|p_{i_nk} $. 
Now imagine that we throw a ball into the initial box, distribute it  uniformly at random to the next generation of boxes, and observe the sequence of the types of the sub-boxes it passes through, generation after generation. We then clearly obtain a version of the Markov chain $Y$, and more generally, throwing $m$ balls independently yields the trajectories of $m$ i.i.d. copies of $Y$.   In this setting, the height $ H_{m,j} $ of  a $(j-1)$-trie corresponds to the first generation at which all boxes contain strictly less than $j$ balls. For   $ j \geq 1 $ the first generation  when there is a box containing strictly  less than $ j $ balls  when $ m $ balls have been thrown  corresponds to $ G_{m,j} $, and is referred to as saturation level.

The random environment that we shall consider corresponds to splitting boxes randomly rather than deterministically, in a Markovian manner. More precisely, we now consider 
for each word ${u}=(i_1, \ldots, i_n)\in \Sigma^n$, an independent  copy $A_{u}=(p_{ik}(u))_{i,k \in\Sigma}$ of a {\it random} transition matrix $A=(p_{ik})_{i,k \in\Sigma}$. The box $b_{{u}}$ with type $i_n$ is split at the next generation into $(b_{{u} k}: k\in \Sigma)$, where $b_{{u} k}$ has type $k$ and size $|b_{{u} k}|=|b_{{u}}|p_{i_nk}(u)$. Then throwing $ m $ balls uniformly at random yields the $m$ data strings in random environment that we are interested in. Note that knowing the family of box sizes, the trajectory of a ball through boxes is not described by a Markov chain. It is  thus not possible to reduce the study to applying the results due to Szpankowski by conditioning on the environment.

In order to investigate the behavior of the heights $H_{m,j}$ and saturation levels $G_{m,j}$ as $m$ tends to infinity, we shall first show how the results by Szpankowski in deterministic environment can be recovered using an occupancy scheme analysis. Even though the main result has already been established  by different arguments, we shall provide a detailed account as the same approach can then  be adapted to the random environment setting. Specifically, we shall investigate the distribution of the sizes of the boxes at a large generation $n$, being especially interested in large deviation type estimates. We will moreover use the same techniques to study  the asymptotics of the height $ H_{m,j} $ when $ j $ depends on $ m $, more precisely $ j=j(m)=m^{\alpha} $ for $ \alpha \in (0,1) $. 

We shall then show that this approach can be adapted in the random environment setting. In this direction, we shall first observe that taking the logarithm of the sizes of boxes yields a multitype branching random walk, which then enables us to apply large deviations estimates due to  \cite{biggins}. We will see that different from the nonrandom case there is a phase transition in the asymptotic behavior of $ H_{m,j} $ as $ m $ tends to infinity and that
\[H_{m,j} \sim  C(j) \cdot  \ln(m) \quad a.s.\]
where $ C(j) $ is a constant depending on $ j $ for $ j $ smaller than a critical parameter and $ C(j)=\zeta^* $ is a constant arising in the limit behavior of the largest box for $ j $ larger than the critical parameter.

The study of the limiting behavior of  the saturation level is closely related to the coupon collector's problem. (See for example \cite{coupon}).
We will consider each box of generation $ n $ as a coupon, each of a different sort and suppose that a collector wants to have at least one of each sort.  We say that a collector buys a certain coupon  if a ball lands in the corresponding box. Then $ G_{m,j} $ is the first generation $ n $ when the collector fails to have at least $ j $ coupons of each sort. We will see that:
\[  G_{m,j} \sim  \zeta_* \cdot  \ln(m) \quad a.s.\]
where $ \zeta_* $ is a constant appearing in the limit behavior of the smallest box. 

Some of our arguments are inspired by the works by  \cite{Jean} and  \cite{Joseph}, who used the fundamental results of  \cite{biggins2} on the asymptotic behaviors of  branching random walks to investigate limits of occupancy scheme in the setting of random multiplicative (monotype) cascades.

\subsection{Large deviation behavior of the box sizes}
In this section we give some results on the asymptotic behavior of the box sizes as the generation $ n $ tends to infinity. We will for simplicity assume that we start from a box of type  $ 1 $.

Recall that $( p_{ij} $, $ i,j \in \Sigma )$ denotes the transition probabilities of the irreducible, aperiodic Markov chain $ Y $. For each $ \theta \in \mathbb{R} $ we define the $ K\times K $-matrix $ A(\theta):= (p_{ij}^{\theta})_{i,j \in \Sigma} $, with the convention that if $ p_{ij}=0$, then $ p_{ij}^{\theta}=0$, even for $ \theta \leq 0$.  
The matrix $ A(\theta) $ is connected to the box sizes in the following way: 
 
 Let $ (l_{i,k}^{(n)})_k $ denote the sequence  of the sizes of boxes with type $ i $ at generation $ n $, lexicographically ordered.  Let us further introduce the point measure
\[Z_j^{(n)}=\sum_k \delta_{-\ln \left(l_{j,k}^{(n)} \right)}, \] 
and the Laplace transform of $ Z_j^{(n)} $, that is:
\[ \mathcal{L}_j^{(n)}(\theta) = \sum_k \left(l_{j,k}^{(n)} \right)^{\theta}. \]
A simple iteration argument then shows, that 

\begin{equation} \label{matrix} 
(\mathcal{L}_1^{(n)}(\theta),...,\mathcal{L}_K^{(n)}(\theta)) = (1,0,...,0)(A(\theta))^n. 
\end{equation}

Recall that we call an eigenvalue $ \rho $ of a matrix $ A $ a maximum modulus eigenvalue, if it is a simple root of the characteristic polynomial and its modulus is strictly larger than the modulus of the other roots.

 Note that since $ Y_n $ is aperiodic and irreducible, $ A(\theta) $ is positive regular (that is its  entries  are finite  and there exists some positive integer $ r $ such that all entries of the matrix $ A(\theta)^r $ are strictly positive.)  Recall that we then have from the Perron-Frobenius theorem:
\begin{enumerate}
\item $A(\theta)  $ possesses a unique maximum modulus eigenvalue $ \rho(\theta) \in \mathbb{R} $,
\item there exists a strictly positive left-eigenvector $ w(\theta)=(w_1(\theta),...,w_K(\theta)) $ and a strictly positive right-eigenvector $ v(\theta)=(v_1(\theta),...,v_K(\theta)) $ with eigenvalue $ \rho(\theta) $, normalized such that we have $ (w(\theta))^tv(\theta)=1 $,
\item  $ \min_i \sum_j A(\theta)_{ij} \leq \rho(\theta) \leq \max_i \sum_j A(\theta)_{ij} $.
\end{enumerate}
We moreover have that:
\begin{lemma}[ \cite{biggins3}] The maximum modulus eigenvalue $ \rho(\theta) $ is analytic in $ \theta$. 
\end{lemma}
We deduce:
\begin{lemma}
\label{asymptotics}
For each $ \theta \in \mathbb{R} $ we have:
\[\lim_{n \rightarrow \infty}  \mathcal{L}_i^{(n)}(\theta)\rho(\theta)^{-n}= v_1(\theta)w_i(\theta)  .\]
\end{lemma}
We shall also need the following:
\begin{proposition}[ \cite{kingman}]
The logarithm of the maximum modulus eigenvalue $ \ln(\rho(\theta)) $ is a convex function of $ \theta $.
\end{proposition}
\begin{remark}
Note that this also entails the convexity of $ \rho(\theta) $.
\end{remark}
Now, introduce the constants
\[ C_*:=\lim_{\theta \rightarrow -\infty } \frac{\rho(\theta)}{-\rho^{\prime}(\theta)}  \quad \textnormal{and} \quad C^*:= \lim_{\theta \rightarrow \infty} \frac{\rho(\theta)}{-\rho^{\prime}(\theta)},  \]
that will play a crucial role in the asymptotic behavior of the smallest and largest box.
\begin{lemma}
We have $ 0< C_* \leq C^* < \infty $.
\end{lemma}
\begin{proof}
First note that the fact that $ \ln(\rho(\theta)) $ is convex entails that $ \frac{\rho(\theta)}{-\rho^{\prime}(\theta)}  $ is an increasing function. Assume first $ \theta >0 $ and let $ p_*:=\inf_{(i,j) \in S} p_{ij} $, where $ S:=\lbrace (i,j) : p_{ij}> 0 \rbrace $.  We then have \[0 < p_*^\theta \leq  \min_i \sum_{j} p_{ij}^\theta \leq \rho(\theta). \]  Now, let $ r $ be such that the matrix $ A^r $ has only positive entries and define the supremum of its entries  $p^*:=\sup_{(i,j)} A^r_{i,j} $. Note that $ p^* < 1 $, that $ A(\theta)^ r $ has only positive entries, and that its maximum modulus eigenvalue is $ \rho(\theta)^r  $. Let $ (l_{ij,k}^{(r)})_k $ denote the sequence  of the sizes of boxes with type $ j $ at generation $ r $ when the first box was of type $ i $. Then:
\[\rho(\theta)^r \leq \max_i \sum_j A(\theta)^r_{ij} = \max_i \sum_j \sum_k \left(l_{ij,k}^{(r)} \right)^{\theta}  \leq K \cdot K^r (p^*)^\theta.\]
Similarly if $ \theta < 0 $ we have $ \rho(\theta) \leq K \cdot p_*^\theta  $ and $ \rho(\theta)^r \geq  (p^*)^\theta $,
and thus:
\[- \infty < \ln(p_*)\leq \lim_{\theta \rightarrow \pm \infty} \frac{\ln(\rho(\theta))}{\theta}   \leq  \frac{\ln(p^*)}{r}< 0. \]
We  get by  l'Hôpitals rule that $-\infty < -\frac{1}{C_*} < 0 $ and $-\infty < -\frac{1}{C^*} < 0 $.
\end{proof}
Now, let us define the size-biased pick of a box of generation $ n $. That is define a random variable $ L_n $ by $ \mathbb{P}(L_n= l_{i,k}^{(n)}) = l_{i,k}^{(n)} $. (Recall that $\sum_{i=1}^K \sum_k l_{i,k}^{(n)}=1. $) We then have:

\begin{lemma}
\label{large deviations}
Let $ \theta \in \mathbb{R} $,  $ z_n(\theta):= e^{n\frac{\rho^{\prime}(\theta)}{\rho(\theta)}} $ and $ \phi(\theta):= \ln(\rho(\theta))-(\theta-1) \frac{\rho^{\prime}(\theta)}{\rho(\theta)} $, then $ -\infty < \phi(\theta) \leq 0 $ for all $ \theta \in \mathbb{R} $, and:

\[\lim_{\varepsilon \downarrow 0} \lim_{n \rightarrow \infty} \frac{1}{n}\ln \left(\mathbb{P} \left( L_n \in  \left(z_n(\theta) e^{-n\epsilon}, z_n(\theta) e^{n\epsilon} \right) \right)\right) = \phi(\theta).\] 

\end{lemma}
\begin{proof}
First note that $\phi(1)=0 $, and $ \phi^\prime(\theta) = (1-\theta) \ln(\rho(\theta))^{\prime\prime} $, thus $ \phi(\theta) $ is decreasing for $ \theta > 1 $ and increasing for $ \theta < 1 $. We then have $ \phi(\theta) \leq 0 $, and the fact that $ -\infty < \phi(\theta) $ follows from the fact that $ \rho $ is analytic.
Now, let $ X_n:=\frac{1}{n}\ln(L_n) $.
We  want to apply the Gärtner-Ellis theorem (see for example \cite{ofer}, Theorem 2.3.6), to the random variable $ X_n $. Let  $ \lambda \in \mathbb{R} $ and define $ \Lambda_n(\lambda) := \ln \mathbb{E}[e^{\lambda X_n}] $. We first need to check that the limit \[\Lambda(\lambda):= \lim_{n \rightarrow \infty} \frac{1}{n} \Lambda_n(n \lambda)   \] exists as an extended real number. Recall that by Lemma 1 we have
\[\frac{K}{2} v(\lambda)(w(\lambda))^t \rho(\lambda)^n \leq \sum_{i=1}^K \mathcal{L}_i^{(n)}(\lambda) \leq  2 K v(\lambda)w^t(\lambda) \rho(\lambda)^n, \]
for sufficiently large $ n $,  and thus 
\begin{eqnarray*}
\lim_{n \rightarrow \infty} \frac{1}{n} \Lambda_n(n 
\lambda) &=& \lim_{n \rightarrow \infty} \frac{1}{n} \ln \left(\sum_{i=1}^K \sum_k  (l_{i,k}^{(n)})^{(\lambda + 1)} \right)  \\&=& \lim_{n \rightarrow \infty} \frac{1}{n} \ln \left(\sum_{i=1}^K \mathcal{L}_i^{(n)}(\lambda + 1) \right)  \\&=& \ln(\rho(\lambda + 1)) \\ &=& \Lambda(\lambda) .
\end{eqnarray*}
We further need to check that $ \Lambda $ is lower semicontinuous and  essentially smooth, that is for \[ \mathcal{D}_{\Lambda}:= \lbrace \lambda \in \mathbb{R} : \Lambda(\lambda) < \infty \rbrace = \mathbb{R}, \] 
the interior  $ \mathcal{D}^o_\Lambda $ is non empty,  $ \Lambda $ is differentiable on $ \mathcal{D}^o_\Lambda $,  and $  \lim_{n \rightarrow \infty} \vert \Lambda^\prime(\lambda_n) \vert = \infty $ for every sequence in $ \mathcal{D}^o_\Lambda $ converging to a boundary point of $ \mathcal{D}^o_\Lambda $.
The essential smothness and lower semicontinuity then follow readily from the fact that $ \Lambda $ is differentiable on $ \mathbb{R} $.
Thus, by the Gärtner-Ellis theorem the large deviation principle holds with good rate function 
 \[\Lambda^*(z)=\sup_{\mu \in \mathbb{R}}( \mu z- \ln(\rho(\mu + 1))).\]
Note that $\Lambda^*(z) \geq 0  $ as $ \rho(1)=1 $. For $ \varepsilon >0 $ let $ B_{\varepsilon}(\theta) := \left(\frac{\rho^{\prime}(\theta)}{\rho(\theta)} -\varepsilon, \frac{\rho^{\prime}(\theta)}{\rho(\theta)} + \varepsilon \right) $ and let $ \overline{B}_{\varepsilon}(\theta)  $ denote its closure.  By the large deviation principle:
 \[\liminf_{n \rightarrow \infty} \frac{1}{n} \ln( \mathbb{P}(X_n \in B_{\varepsilon}(\theta))) \geq -\inf_{x \in B_{\varepsilon}(\theta) } \Lambda^*(x) \]
 and  

\[\limsup_{n \rightarrow \infty} \frac{1}{n} \ln( \mathbb{P}(X_n \in B_{\varepsilon}(\theta))) \leq  -\inf_{x \in \overline{B}_{\varepsilon}(\theta)} \Lambda^*(x). \]
Now, note that $ -\Lambda^*\left(\frac{\rho^{\prime}(\theta)}{\rho(\theta)} \right)  = \phi(\theta) $ and thus:
\begin{eqnarray*}
\phi(\theta) & \leq & \liminf_{\varepsilon \downarrow 0} \liminf_{n \rightarrow \infty} \frac{1}{n} \ln( \mathbb{P}(X_n \in B_{\varepsilon}(\theta))) \\ & \leq & \limsup_{\varepsilon \downarrow 0} \limsup_{n \rightarrow \infty} \frac{1}{n} \ln( \mathbb{P}(X_n \in B_{\varepsilon}(\theta)))   \leq  \phi( \theta).
\end{eqnarray*}
This proves our  claim.
\end{proof}

We next provide  another approach to the previous lemma. Imagine that we follow the trajectory of a ball through the nested sequence of boxes and recall that the sequence of types of the boxes the ball passes through corresponds to a trajectory of the Markov chain $ Y $.  Now, let $ T_n $ denote the type of the box of the $  n$-th generation the ball is passed through and note that $ L_n $ defined in the previous Lemma corresponds to its size.
Then $ (T_n) $ is  an irreducible Markov chain with transition probabilities $ p_{ij} $ and  $ (T_n,T_{n+1}) $ is an irreducible Markov chain with state space $ \lbrace (i,j) \in \Sigma \times \Sigma : p_{ij} \neq 0 \rbrace $ and transition probabilities $ \Pi((i,j),(j,k))=p_{ij} $ and $ \Pi((i,j),(m,k))=0 $ if $ m \neq j $. Further we have that
 \[L_n= \prod_{k=0}^{n-1} p_{T_k,T_{k+1}}. \]
 We  prove again Lemma \ref{large deviations} now   using the large deviation principle for additive functionals of Markov chains.
 
 \begin{proof} Define the deterministic function \[f: \Sigma \times \Sigma \rightarrow \mathbb{R}, \quad (i,j) \mapsto p_{ij}, \]
 and the empirical means 
 \[X_n:= \frac{1}{n} \sum_{k=0}^{n-1} f(T_k,T_{k+1}) = \frac{1}{n} \sum_{k=0}^{n-1} \ln(p_{T_k,T_{k+1}})=\frac{1}{n}\ln(L_n). \] By Theorem 3.1.2 in \cite{ofer}, $ X_n $ fulfills the large deviation principle with good rate function
 \[\Lambda^*(z)=\sup_{\mu \in \mathbb{R}}( \mu z- \ln(\rho(\mu + 1))).\]
The claim then follows in the same spirit as before.
 \end{proof}
 
Let us define a function that will play a crucial role in our analysis:
\begin{eqnarray}
\psi(\theta):= \ln(\rho(\theta)) - \frac{\rho^{\prime}(\theta)}{\rho(\theta)}\theta
\end{eqnarray}
We then have:
 \begin{cor} 
 
 \label{numerous}
 
 Let $ \theta \in \mathbb{R} $ and define  $ z_n(\theta):= e^{n\frac{\rho^{\prime}(\theta)}{\rho(\theta)}} $. We then have:
\[\lim_{\epsilon \downarrow 0} \lim_{n \rightarrow \infty} \frac{1}{n} \ln \left( \sum_{i=1}^K \sum_k 1_{ \left\lbrace l^{(n)}_{i,k} \in \left(z_n(\theta) e^{-n\epsilon}, z_n(\theta) e^{n\epsilon} \right) \right\rbrace} \right)=  \psi (\theta). \]

 \end{cor}
 \begin{proof}
 First note that
  \[\mathbb{P} \left( L_n \in  \left(z_n(\theta) e^{-n\epsilon}, z_n(\theta) e^{n\epsilon} \right) \right)= \sum_{i=1}^K \sum_k 1_{ \left\lbrace l^{(n)}_{i,k} \in \left(z_n(\theta) e^{-n\epsilon}, z_n(\theta) e^{n\epsilon} \right) \right\rbrace} l_{i,k}^{(n)}.\]
Thus
  \begin{eqnarray*}
\sum_{i=1}^K   \sum_k 1_{ \left\lbrace l^{(n)}_{i,k} \in \left(z_n(\theta) e^{-n\epsilon}, z_n(\theta) e^{n\epsilon} \right) \right\rbrace} e^{-n\epsilon} & \leq & z_n(\theta)^{-1} \cdot \mathbb{P} \left( L_n \in  \left(z_n(\theta) e^{-n\epsilon}, z_n(\theta) e^{n\epsilon} \right) \right) \\ & \leq & \sum_{i=1}^K \sum_k 1_{ \left\lbrace l^{(n)}_{i,k} \in \left(z_n(\theta) e^{-n\epsilon}, z_n(\theta) e^{n\epsilon} \right) \right\rbrace}  e^{n\epsilon}.
  \end{eqnarray*}
By Lemma \ref{large deviations} we  derive that
\[\limsup_{n \rightarrow \infty}  \frac{1}{n} \ln \left( \sum_{i=1}^K \sum_k 1_{ \left\lbrace l^{(n)}_{i,k} \in \left(z_n(\theta) e^{-n\epsilon}, z_n(\theta) e^{n\epsilon} \right) \right\rbrace} \right) \leq \psi(\theta) + \epsilon, \]
and \[ \liminf_{n \rightarrow \infty} \frac{1}{n} \ln \left( \sum_{i=1}^K \sum_k 1_{ \left\lbrace l^{(n)}_{i,k} \in \left(z_n(\theta) e^{-n\epsilon}, z_n(\theta) e^{n\epsilon} \right) \right\rbrace} \right) \geq \psi(\theta) - \epsilon\]
and we easily  conclude.
 \end{proof}
 
We moreover have that:
 \begin{lemma} We have that $ \psi(\theta) > 0 $ for all $ \theta \in \mathbb{R} $.
 \end{lemma}
 \begin{proof}
First note that Corollary 1 implies that $ \psi(\theta) \geq 0  $ for all $\theta \in \mathbb{R}  $. Further we have that $ \psi(0) > 0$ since $ \rho(0) > 1$.  Moreover $ \psi^{\prime}(\theta)=-\theta \cdot \ln(\rho(\theta))^{\prime \prime}  $ and thus, by convexity of $ \ln(\rho(\theta)) $,  we have that $ \psi $ is increasing on the interval $ (-\infty,0) $ and decreasing on the interval $ (0, \infty) $. Together with the fact that $ \psi  $ is analytic and thus cannot be zero on any interval, we arrive at $ \psi(\theta) >0 $ for all $ \theta \in \mathbb{R} $.
\end{proof}

We can use the previous  results to gain information about the asymptotic behavior of the sizes of the smallest box and the largest box. 
\begin{lemma}
Let $ \underline{l}^{(n)} $ denote the size of the smallest box at generation $ n $. We then have:
\[\lim_{n \rightarrow \infty} \frac{\ln (\underline{l}^{(n)})}{n}= -\frac{1}{C_*}\]
\end{lemma}
\begin{proof}
 Let  $ \theta < 0 $ .
By Lemma 1  there exists a $ n_0 $, such that for all $ n \geq n_0 $:
\[(\underline{l}^{(n)})^{\theta} \leq \mathcal{L}_i^{(n)}(\theta) \leq 2 v_1(\theta)(w_i(\theta))^t \rho(\theta)^n. \] 
By some rearrangement we thus get:
\[\liminf_{n \rightarrow \infty} \frac{\ln (\underline{l}^{(n)})}{n}\geq \frac{\ln(\rho(\theta))}{\theta},\]
and we conclude by letting $ \theta $ tend to $ -\infty $. Now let $ \varepsilon > 0 $. By Lemma \ref{large deviations}  there exists a natural number $ n_0(\varepsilon) $ such that $ \underline{l}^{(n)} \leq e^{n\frac{\rho^{\prime}(\theta) } {\rho(\theta)}}e^{\varepsilon n} $ for all $ n \geq n_0(\varepsilon) $, thus 
\[\limsup_{n \rightarrow \infty} \frac{\ln(\underline{l}^{(n)})}{n}\leq \frac{\rho^{\prime}(\theta) } {\rho(\theta)} + \varepsilon, \]
and the result follows by letting $ \varepsilon $ tend to zero.

\end{proof}

In the same way one shows for the largest box that:
\begin{lemma}
Let $ \overline{l}^{(n)} $ denote the size of the largest box at generation $ n $. We then have:
\[\lim_{n \rightarrow \infty} \frac{\ln (\overline{l}^{(n)})}{n}= -\frac{1}{C^*}\]
\end{lemma}
 
\section{Height and saturation level of Markovian tries}
\subsection{Height of Markovian tries}
\label{Upper Bound}
Recall that   $ H_{m,j}$ for $ j \geq 2 $,  denotes the first generation of boxes at which all the boxes contain strictly less than $ j $ balls when $ m $  balls have been thrown independently.
We shall show how the large deviation estimates of the preceding section enable us to recover the following result.
\begin{theorem}[ \cite{S}] For every $ j \geq 2 $ we have
\[
\lim_{m \rightarrow \infty} \frac{1}{\ln(m)}  H_{m,j} = \frac{j}{-\ln(\rho(j))}  \qquad a.s.  \]

\end{theorem}

Note that we  now  continue to successively throw balls forever, and that the randomness is coming from the way we distribute the balls into the boxes.

 We will first tackle the upper bound. 
 Let $ N^{(n)}_{m,j} $ denote the number of boxes at generation $ n $ containing  $ j $ or more balls when $ m $ balls have been thrown. Note that  $ H_{m,j} < n $ when $ N^{(n)}_{m,j}=0 $. The idea will  be to analyze the asymptotic behavior of $ N_{m,j}^{(n)} $ as $ n $ and $ m $ tend to infinity. We will show that \[  H_{m,j} \leq  \frac{j}{-\ln(\rho(j))} \ln(m) +\textit{O}\left( \ln \ln(m) \right) \quad a.s., \] as $ m $ tends to infinity. Let $  a>\frac{1}{j} $ and define the sequence
 \[x_n:=\rho(j)^{-\frac{n}{j}}n^{-a}. \]
We then have:
\begin{lemma}
\label{poissonboxes}
For almost all $ \omega $, there exists a natural number $ n_0(\omega) $ s.t for all $ n \geq n_0(\omega) $, $ N^{(n)}_{ \lfloor x_n \rfloor,j} =0 $.
\end{lemma}
\begin{proof}
Let  $B(m, p)$ denote a generic Binomial variable with parameter $ p \in [0,1] $ and $ m \in \mathbb{N} $ and note that
\begin{eqnarray}
\label{Binomial}
 \mathbb{P}(B(m, p) \geq j) \leq  \left(\begin{matrix}
m \\j 
\end{matrix} \right) p^j  \leq m^j  p^j. \end{eqnarray}
Further note that the number of balls in a box of size $ l $ when $ m  $ balls have been thrown is $ B(m,l) $ distributed.
We thus have
\begin{eqnarray*}
\mathbb{E}[ N^{(n)}_{ \lfloor x_n \rfloor,j} ] &=& \sum_{i=1}^K \sum_k \mathbb{P} \left( B \left(\lfloor x_n \rfloor, l_{i,k}^{(n)} \right) \geq j \right) \\ & \leq & \sum_{i=1}^K \sum_k \left(  \lfloor x_n \rfloor l_{i,k}^{(n)} \right)^j
\\  & \leq & \sum_{i=1}^K   x_n^j \mathcal{L}_i^n(j). \end{eqnarray*}
By Lemma \ref{asymptotics} there exists a natural number $ n_1 $ such that for all $ n \geq n_1 $
\[  \mathcal{L}_i^{(n)}(j)\rho(j)^{-n} \leq 2v_1(j)w_i(j)^t. \]
Taking $  c_1(j):=2v_1(j)w_i(j)^t $ we get that
\begin{eqnarray*}
\mathbb{E}[ N^{(n)}_{ \lfloor x_n \rfloor,j} ] \leq \sum_{i=1}^K  c_1(j) x_n^j \rho(j)^n   \leq \sum_{i=1}^K c_1(j)  n^{-aj},
\end{eqnarray*}
for all $ n \geq n_1$. We finally arrive at
\begin{eqnarray*}
\mathbb{E} \left[ \sum_{ n \geq n_1} 1_{ \lbrace N^{(n)}_{ \lfloor x_n \rfloor,j} \geq 1 \rbrace} \right] \leq \mathbb{E} \left[\sum_{ n \geq n_1}  N^{(n)}_{ \lfloor x_n \rfloor,j} \right] \leq  \sum_{ n \geq n_1} c_1(j)  n^{-aj} < \infty  ,
\end{eqnarray*}
 and we conclude by the Borel-Cantelli lemma.
\end{proof}

This lemma will be enough to show that:
\begin{proposition}

\label{upper bound}
For every integer $ j \geq 2 $ we have
\begin{eqnarray*}   \frac{1}{\ln(m)}  H_{m,j} \leq \frac{j}{-\ln(\rho(j))} +\textit{O}\left( \frac{\ln \ln(m)}{\ln(m)} \right)\qquad a.s. \end{eqnarray*}
as $ m $ tends to infinity.
\end{proposition}
\begin{proof}
First note that there exists a natural number $ n_1 $ such that $ \forall n \geq n_1 $ we have $ n^{-a} \geq \exp(-n  \frac{-\ln(\rho(j))}{2j}) $ and there exists a natural number $ n_2 $ such that for all $ n \geq n_2 $ the sequence $ x_n $ is increasing. Then using Lemma \ref{poissonboxes} choose an  $ \omega \in \Omega $ for which there exists a natural number $ n_0(\omega) $ such that for all $ n \geq n_0(\omega) $ we have $ N^{(n)}_{ \lfloor x_n \rfloor,j} =0 $.  Let $ n_3 \geq \max(n_0(\omega), n_1+1, n_2+1) $ and note that for each $  m \geq x_{n_3} $ there exists an unique $ n \geq n_3 $ such that $ x_{n-1} \leq m < x_n $. From $  m \leq \lfloor x_n \rfloor$ we have $ H_{m,j} \leq n $. Moreover taking logarithm on both sides of the inequality $ x_{n-1} \leq m $, we derive that
\[n \leq \frac{j}{- \ln(\rho(j))}\ln(m)+  \frac{j}{- \ln(\rho(j))} a \ln(n-1)+1. \]
Since $ n \geq n_1+1 $, we have that 
\[m \geq x_{n-1} = \exp \left((n-1)  \frac{-\ln(\rho(j))}{j} \right) (n-1)^{-a} \geq \exp \left((n-1) \frac{-\ln(\rho(j))}{2j} \right)  .\]
Taking logarithm, we derive that 
\[(n-1) \leq \frac{2j}{-\ln(\rho(j))}\ln(m)\]
and  finally arrive at:
\[H_{m,j} \leq n \leq \frac{j}{- \ln(\rho(j))}\ln(m)+  \frac{j}{- \ln ( \rho(j))} a \ln \left( 2\frac{j}{-\ln(\rho(j))}\ln(m) \right)+1, \]
and we easily conclude.
\end{proof}
\begin{remark}
One could also have derived Proposition \ref{upper bound}, by applying Theorem 6.B in  \cite{Barbour}.
\end{remark}

We now turn to the proof of the lower bound.
Recall that $ H_{m,j} \geq n $, if at generation $ n $ there is at least one box containing  $ j $ or more balls. We thus have that $ H_{m,j} > n  $ if $ N_{m,j}^{(n)} \geq 1 $. As in the proof of the upper bound we want to analyze the asymptotic behavior of $ N_{m,j}^{(n)} $ as $ n $ and $ m $ tend to infinity. We want to take $ n= \frac{j}{-\ln(\rho(j))} \ln(m) +\textit{o}\left( \ln(m) \right) $. Recall that $ \psi(\theta)= \ln(\rho(\theta)) - \frac{\rho^{\prime}(\theta)}{\rho(\theta)}\theta $ and that $ \psi(\theta)>0 $. Let $  \psi(j) > \varepsilon^{\prime} >0$ and define the sequence:
 \[ x_n:= \rho(j)^{\frac{-n}{j}}e^{n  \varepsilon^{\prime}}.  \]

 In this section we  make use of the classical Poissonization trick. Instead of throwing  $ x_n $ balls initially we will throw the random number of balls $ \Poisson(x_n) $. The advantage of this procedure is that if we consider two different boxes $ b $ and $ \tilde{b} $ with size $ l $ respectively $ \tilde{l} $, then the number of balls in $ b $ and $ \tilde{b} $ are independent Poisson random variables with parameter $ x_nl $ respectively $ x_n\tilde{l} $.

\begin{lemma}
\label{tuttituu}

 For almost all $ \omega $ there exists a natural number $ n_0(\omega) $, such that there exists at least one box at generation $ n $ containing $ j $ or more balls when $ \Poisson(x_n) $ balls have been thrown. 
\end{lemma}
\begin{proof} From Corollary \ref{numerous}, we know that  for all $ \varepsilon > 0 $ there exists a natural number $ n_1(\varepsilon) $ such that
 \begin{equation} \label{schwubedibummm} \frac{1}{n} \ln \left( \sum_{i=1}^K \sum_k 1_{ \left\lbrace l^{(n)}_{i,k} \in \left(e^{n\frac{\rho^{\prime}(j)}{\rho(j)}} e^{-n\varepsilon}, e^{n\frac{\rho^{\prime}(j)}{\rho(j)}} e^{n\varepsilon} \right) \right\rbrace} \right) \geq  \psi(j) - \varepsilon , \end{equation}
 for all $ n \geq n_1(\varepsilon) $. Now, let $ \varepsilon <  \frac{j}{j+1} \varepsilon^{\prime} $ and let $ M_n $ denote the set containing all the boxes with size larger than  $   z_n:= e^{n \left( \frac{\rho^{\prime}(j)}{\rho(j)}-\varepsilon \right)}  $ at generation $ n $. From (\ref{schwubedibummm}) we deduce that \[ \vert M_n \vert \geq v_n(j) := e^{n(  \psi(j)- \varepsilon)}, \] 
 for all $ n \geq n_1(\varepsilon) $.
 For $ n $ large enough we can thus consider the first $ v_n $ boxes in $ M_n $, say $ b_1(n),...,b_{v_n(j)}(n) $ and denote their size with $ l_1(n),...,l_{v_n(j)}(n) $. We then place an imaginary box  $ \textbf{b}_i(n) $ in $ b_i(n) $ for $ 1\leq i \leq v_n(j) $, each of size exactly $ z_n $. When a ball falls into the box $ b_i(n) $ it arrives in the imaginary box $ \textbf{b}_i(n) $ with probability $ \frac{z_n}{ l_i} $.  We want to show that there exists a natural number $ n_0 $ such that for all $ n \geq n_0 $ at least one of the boxes $\textbf{b}_i(n) $, $ 1\leq i \leq v_n(j) $ contains more than $ j $ balls when Poisson($ x_n $) balls have been thrown.

 Let  $ A_n $ denote the event that all boxes $\textbf{b}_i(n) $, $ 1\leq i \leq v_n(j) $ contain strictly less then $ j $ balls when $ \Poisson(x_n) $ balls have been thrown. Since the number of balls in each box are independent Poisson  random variables with parameter $ x_nz_n $, we have:
\[\mathbb{P}(A_n) \leq \mathbb{P}(\Poisson(x_nz_n) < j)^{v_n(j)} \]
and thus
\[\ln(\mathbb{P}(A_n)) \leq v_n(j) \cdot \ln(\mathbb{P}(\Poisson(x_nz_n) < j)).  \]
 Note that since $ \psi(j) > 0 $, the sequence $ x_nz_n $ tends to zero as $ n $ tends to infinity and  that $ \ln(1+x) \sim x $ for small $ x $. We get that for $ n $ large enough
\[ \ln(\mathbb{P}(\Poisson(x_nz_n) < j)) \sim -\frac{(x_nz_n)^j}{j!}, \] 
and thus for large enough $ n $ there exists a constant $ c $ such that:
\[ \ln(\mathbb{P}(A_n)) \leq -cv_n(j)x_n^jz_n^j= -ce^{n(j \varepsilon^{\prime}-(j+1)\varepsilon)}, \]

and we finally arrive at
\[ \mathbb{P}(A_n) \leq e^{   -ce^{n(j \varepsilon^{\prime}-(j+1)\varepsilon)}}. \]
Applying the  Borel-Cantelli lemma, we get the result.
\end{proof}

This lemma will be enough to show the following proposition:
\begin{proposition}
\label{propediprop}
 For every integer $ j \geq 2 $ we have
\begin{eqnarray*} \liminf_{m\rightarrow \infty} \frac{1}{\ln(m)}  H_{m,j} \geq \frac{j}{-\ln(\rho(j))} \quad  a.s. \end{eqnarray*}

\end{proposition}
\begin{proof}
Define the sequence $ y_n= \rho(j)^{- \frac{n}{j}}e^{n  \varepsilon}  $ with $ \varepsilon > \varepsilon^{\prime} $ and notice that the sequence $ \frac{y_{n-1}}{x_n} $ tends to infinity. Thus there exists a natural number $ n_1 $ such that, for all $ n \geq n_1 $, $ \frac{y_{n-1}}{x_n}>3 $. We have:
\begin{eqnarray*}
\mathbb{P}(\Poisson(x_n) \geq \lceil y_{n-1} \rceil) &\leq & \mathbb{P}(\Poisson(x_n)\geq 3x_n) \\ &=& \mathbb{P}(\Poisson(x_n)-x_n \geq 2x_n) \\ & \leq & \mathbb{P}(\vert \Poisson(x_n)-x_n \vert \geq 2x_n) \leq\frac{1}{4x_n}
\end{eqnarray*}
by Chebyshev's inequality. By Borel-Cantelli's lemma we thus derive that for almost all $ \omega  $ there exists a natural number  $ n_2(\omega) $ such that for all $ n \geq n_2(\omega) $ we have that $ \Poisson(x_n) < \lceil y_{n-1} \rceil $. Now by Lemma \ref{tuttituu} we derive that for almost all $ \omega $ there exists a natural number $ n_3(\omega) $ such that for all $ n \geq n_3(\omega) $, we have:
\[N_{ \lceil y_{n-1} \rceil, j } \geq 1. \]
 Now fix such an $ \omega $. Note that  there exists a natural number $ n_4 $ such that $ \forall n \geq n_4 $ the sequence $ y_n $ is increasing. Let $ n_5 \geq \max(n_3(\omega), n_4+1) $ and note that for each $ m \geq y_{n_4} $ there exists a unique $ n \geq n_4 $ such that $ y_{n-1} < m <y_n $. Now, since $ \lceil y_{n-1} \rceil \leq m  $, we have $ H_{m , j} >n $. Further since $  m < y_n $, taking logarithm on both sides, we have that:
\[ \ln(m) < \left(\frac{- \ln(\rho(j))}{j} +\varepsilon \right) n.\]

We thus have that:
\[ \left(\frac{- \ln(\rho(j))}{j} +\varepsilon \right)^{-1} \ln(m) < n < H_{m,j},\]
and 
\[ \frac{1}{m} H_{m,j} \geq \left(\frac{- \ln(\rho(j))}{j} +\varepsilon \right). \]
We conclude by letting $ \varepsilon  $ tend to zero.
\end{proof}
\subsection{Power regimes}
In this section we  study the case when $ j=j(m)=m^{\alpha} $ for $ \alpha \in (0,1) $. In the setting of $ K $-ary tries this corresponds to the case when not only the number of words that have to be stored, but also the storage capacity of the nodes tends to infinity. We aim to show that:
\begin{theorem}
  We  have:

\begin{eqnarray*} \lim_{m \rightarrow \infty}  \frac{1}{\ln(m)}  H_{m,j} = (1-\alpha) C^* \qquad a.s. \end{eqnarray*}
\end{theorem}

We first establish the upper bound. Let  $ \theta > 0 $, $ a>\frac{1}{\theta(1-\alpha)} $, and define the sequence \[x_n:=\rho(\theta)^{-\frac{1}{1-\alpha}\frac{n}{\theta}}n^{-a}. \]  In the notation of  the last section, we then have:
\begin{lemma}
\label{sdftewerg}

For almost all $ \omega $, there exists a natural number $ n_0(\omega) $ s.t for all $ n \geq n_0(\omega) $ we have $ N^{(n)}_{ \lfloor x_n \rfloor,x_{n-1}^{\alpha}} =0 $.
\end{lemma}
\begin{proof}
Let $k \in \mathbb{N} $ and let $ j \geq 2(k+1) $. A straight forward computation shows that then:
\[\mathbb{P}(B(m,p) \geq j) \leq 2^{k} \cdot \left(\frac{mp}{j}\right)^k.\]

Thus for $ n $ large enough such that $ x_{n-1}^{\alpha} \geq 2 (\lceil \theta \rceil +1) $, we have::
\begin{eqnarray*}
\mathbb{E}[ N^{(n)}_{ \lfloor x_n \rfloor,x_{n-1}^{\alpha}} ] & \leq & \sum_{i=1}^K \sum_k \mathbb{P} \left( B \left(\lfloor x_n \rfloor, l_{i,k}^{(n)} \right) \geq \lceil \theta \rceil \right) \\ & \leq & 2^{\lceil \theta \rceil} \sum_{i=1}^K \sum_k \left(   \frac{\lfloor x_n \rfloor l_{i,k}^{(n)}}{x_{n-1}^{\alpha}} \right)^{\theta}.
 \end{eqnarray*}
Let $ c_1(\theta)= 2^{\lceil \theta \rceil} \cdot 2v_1(\theta)(w_i(\theta))^t $. By Lemma \ref{asymptotics} there exists a natural number $ n_1 $ such that for all $ n \geq n_1 $:
 \begin{equation}
 \mathbb{E}[ N^{(n)}_{ \lfloor x_n \rfloor,x_{n-1}^{\alpha}} ] \leq  \sum_{i=1}^K c_1(\theta) \left( \frac{x_n}{x_{n-1}^{\alpha}} \right)^{\theta} \rho(\theta)^n \leq \sum_{i=1}^K c_2(\theta) n^{-a \theta (1-\alpha)},
 \end{equation}
where $ c_2(\theta)=c_1(\theta)\rho(\theta)^{-\frac{\alpha}{1-\alpha}} $, and we conclude by Borel-Cantelli's lemma.
\end{proof}

The upper bound then follows in the same way as Proposition \ref{upper bound}.
We now turn to the proof of the lower bound. Let $ \varepsilon^{\prime} >0 $ and define the sequence: 
\[y_n:= \exp\left( n \left( \frac{1}{C^*}\frac{1}{(1-\alpha)} + \varepsilon^{\prime} \right)\right).\]
We then have:
\begin{lemma}
\label{Lemma6}
 For almost all $ \omega \in \Omega $ there exists a natural number $ n_0(\omega) $, s.t. $ N_{\lceil y_{n-1} \rceil ,y_n^{\alpha}} \geq 1 $ for all $ n \geq n_0(\omega) $.
\end{lemma}
\begin{proof}
Let $ \varepsilon < (1 - \alpha) \varepsilon^{\prime} $. Define $ A_n $, the event that the largest box at generation $ n $ contains strictly less than $ y_n^{\alpha} $ balls when $  \lceil y_{n-1} \rceil $ balls have been thrown and recall that  the largest box at generation $ n $ has size larger or equal to $ z_n:= e^{n (-\frac{1}{C^*} - \varepsilon )} $. We then have:
\begin{eqnarray*} \mathbb{P}(A_n)  & \leq &  \mathbb{P}((B(\lceil y_{n-1}  \rceil, z_n ) > y_n^{\alpha}) \\ & \leq & \mathbb{P}((B(\lceil y_{n-1}  \rceil, z_n )+1)^{-1} \leq y_n^{- \alpha}). \end{eqnarray*}
A straight forward computation shows that:
\[ \mathbb{E}[(B(m,p)+1)^{-1}]=\frac{1-(1-p)^{m+1}}{(m+1)p} \leq \frac{1}{mp},\]
and by Markov's  inequality we thus get that
\begin{eqnarray*} \mathbb{P}((B(\lceil y_{n -1} \rceil, z_n )+1)^{-1} \leq y_n^{- \alpha}) &\leq & y_n^{\alpha} \cdot \mathbb{E}[(B(\lceil y_{n -1} \rceil, z_n )+1)^{-1}] \\ &\leq &  \frac{y_n^{\alpha}}{\lceil y_{n-1} \rceil \cdot z_n                                                                                                                                                     }.
\end{eqnarray*}
Taking $ c(\alpha):=y_1^{\alpha}$ we arrive at:
\[\mathbb{P}(A_n ) \leq c(\alpha) \cdot e^{(n-1) (\varepsilon + (\alpha -1) \varepsilon^{\prime})}. \]
and we conclude by Borel-Cantelli's Lemma. 
\end{proof}

The lower bound follows by the usual computations.

 \subsection{Saturation level of Markovian tries}

 Let $ j\geq 1 $ and recall that $ G_{m,j} $ denotes the first generation at which there exists a box containing strictly less than $ j $ balls, when $ m $ balls have been thrown initially. We aim to show that:
\begin{theorem}

 Let $ j\geq 1 $, and recall that $ C_*=\lim_{\substack{\theta \rightarrow -\infty }} \frac{\rho(\theta)}{-\rho^{\prime}(\theta)} $. We then have
\[
\lim_{m \rightarrow \infty} \frac{1}{\ln(m)}  G_{m,j} = C_*  \qquad a.s.  \]

\end{theorem}

 We will first study the asymptotic behavior of $ G_{m,1} $ and then extend this to the asymptotic behavior of $ G_{m,j} $.  We will shorthand write $ G_m $ for $ G_{m,1} $.

In order to establish an upper and a lower bound for $ G_m $ it will be useful to study the asymptotic behavior of the number of balls $ T_n $, one needs to throw initially to observe at least one ball in each box of generation $ n $. 

As already mentioned in the introduction, this can be interpreted in terms of the coupon collector's problem. We will see each box of generation $ n $ as a different sort of coupon,  $ T_n $ then corresponds to the number of coupons one has to buy to get at least one of each coupon and the probability to buy a sort of coupon is given by the size of the corresponding box.

It will sometimes  be convenient to use a Poissonization technique to model $ T_n $. Suppose that the collector continues to buy coupons forever (rather than stopping when having a full collection). Moreover suppose that the coupons are bought at times distributed as the arrival times of  a Poisson process with rate $ 1 $. The times a coupon of sort $ i $ is bought are then  the arrival times of independent Poisson processes with rate $ p_i $. Let $ \textbf{T}_n $ denote the waiting time until the collector has completed his collection. We then have \[\textbf{T}_n= \max_{k \leq K^n} \textbf{exp}(l_k^{(n)}), \] where $\textbf{exp}(l_k^{(n)})$ denote independent exponential random variables with parameter $ l_k^{(n)} $, that is to say $ \mathbb{E}[\textbf{exp}(l_k^{(n)})]=\frac{1}{l_k^{(n)}} $. The connection between $ \textbf{T}_n $ and $ T_n $ is then given by 
\[\textbf{T}_n= \sum_{k=1}^{T_n} \textbf{exp}_k(1), \]
where $\textbf{exp}_k(1)  $ are i.i.d. exponential random variables with parameter $ 1 $, independent of $ T_n $. In the same spirit,  let $ T_n^j $ denote the number of balls  one needs to throw initially to observe at least $ j $ ball in each box of generation $ n $. 
Let $ \textbf{T}_n^j $ denote the waiting time until the collector has completed $ j $ copies of his collection. We then have \[\textbf{T}^j_n= \max_{k \leq K^n} \Gamma(j,l_k^{(n)}), \] where $\Gamma(j,l_k^{(n)})$ denote independent Gamma random variables with parameter $ l_k^{(n)} $ and $ j $. As in the case $ j=1 $ we then have
\[\textbf{T}_n^j= \sum_{k=1}^{T_n^j} \textbf{exp}_k(1), \]
where $\textbf{exp}_k(1)  $ are i.i.d. exponential random variables with parameter $ 1 $, independent of $ T_n^j $.

For the lower bound, we want to find an upper bound for  the waiting time $ \textbf{T}_n $ until the collector has completed his collection.

 \begin{lemma}
  \label{markov}

 Let $ \theta < 0 $ and define $ x_n:=e^{n\frac{\ln((\rho(\theta))}{-\theta}}n^{\mu}  $ , $ \mu >2 $. For almost all $ \omega $ there exists a natural number $ n_0(\omega) $, such that for all $ n \geq n_0(\omega) $ :
 \[\textbf{T}_n(\omega) < x_n. \]
 
 \end{lemma}

 \begin{proof} 
We use the  Poissonization technique explained in the previous section. We have: 
 \[\mathbb{P}(\textbf{T}_n \geq x_n) = \mathbb{P}(\max_{k \leq K^n} \textbf{exp}(l_k^{(n)})  \geq x_n) \leq \mathbb{P}( \max_{k \leq K^n} \textbf{exp}_k(1) \geq \underline{l}^{(n)}x_n),\]
 where $ \textbf{exp}_k(1) $ are i.i.d exponential r.v. with parameter $ 1 $. Now, for  $ \theta < 0 $ we have:
\[\left( \underline{l}^{(n)} \right)^{\theta} \leq \sum_k \left(l_{j,k}^{(n)} \right)^{\theta}  =  \mathcal{L}^{(n)}_j(\theta). \]
By Lemma 1  there exists a $ n_0 $, such that for all $ n \geq n_0 $:
\[(\underline{l}^{(n)})^{\theta} \leq \mathcal{L}_j^{(n)}(\theta) \leq 2 v_1(\theta)(w_i(\theta))^t \rho(\theta)^n. \]
Let $ c_1(\theta):=\frac{1}{\theta}2\ln(v_1(\theta)(w_i(\theta))^t $, by taking logarithm on both sides and some rearrangement, we  get that there exists a natural number $ n_1 $ such that for all $ n \geq n_1 $:
\[ \frac{n \ln(\rho(\theta))}{\theta} + c_1(\theta) \leq \ln (\underline{l}^{(n)}),  \]
and thus 
\[ e^{n\frac{\ln(\rho(\theta))}{\theta} + c_1(\theta)} \leq   \underline{l}^{(n)}, \]
for all $ n \geq n_1 $. Let $ c_2(\theta) = e^{c_1(\theta)} $, we deduce that for $ n $ large enough:
 \[\mathbb{P}(\textbf{T}_n \geq x_n) \leq \mathbb{P}( \max_{k \leq K^n} \textbf{exp}_k(1) \geq c_2(\theta)n^{\mu}).\]

 Now, recall that  \[ \max_{k \leq K^n}  \textbf{exp}_k(1) - n\ln(K)\]
converges in distribution to the standard Gumbel distribution $ \textbf{G}(1) $ as $ n $ tends to infinity. Thus for large enough $ n $, we have:
\begin{eqnarray*} \mathbb{P}(\textbf{T}_n \geq x_n) & \leq & 2 \cdot \mathbb{P}\left(\textbf{G}(1) \geq c_2(\theta)n^{\mu}-n\ln(K)\right) \\ &=& 2 \cdot \left( 1-\exp\left(-e^{-(c_2(\theta)n^{\mu}-n\ln(K))}\right) \right) \\  & \leq &  2 \cdot e^{-(c_2(\theta)n^{\mu}-n\ln(K))}.\end{eqnarray*}
 We then conclude by Borel-Cantelli's lemma.
 \end{proof}

 Since the times at which a ball is thrown are the arrival times of independent Poisson processes with rate $ 1 $, the number of balls thrown up to time $ x_n $ is Poisson distributed with parameter $ x_n $. Applying Lemma \ref{markov}, we thus get, that:
\begin{cor}
\label{zero boxes}
 For almost all $ \omega $ it exists an integer number $ n_0(\omega) $ such that $ \forall n \geq n_0(\omega)$ there exists no box at generation $ n $ containing no balls when $ \Poisson(x_n) $ balls have been thrown.
\end{cor}
We then arrive at:
\begin{proposition}
\label{Upper Bound saturation}
We have:
\[ \frac{1}{\ln(m)} G_{m,j} \geq C_*+ \textit{O}\left(\frac{\ln\ln(m)}{\ln(m)}\right) \quad a.s. \]
as $ m $ tends to infinity.
\end{proposition}

Note that $T_n^j \leq T_n^{(1)}+...+T_n^{(j)}$ a.s.,
where $T_n^{(1)},...,T_n^{(j)}  $ denote independent copies of $ T_n $. Further we have $ G_{\lfloor \frac{m}{j} \rfloor } \geq  n$ a.s. which implies that $ \lfloor \frac{m}{j} \rfloor > T_n $ a.s. and thus $ T_n^j \leq m $ a.s. But then $ G_{m,j} \geq n $ a.s. and we conclude that $ G_{m,j} \geq G_{\lfloor \frac{m}{j} \rfloor} $ a.s. It will be thus enough to show the statement for $ G_m $. 

 \begin{proof}[of Proposition \ref{Upper Bound saturation}]
Define the sequence $ y_n= e^{n\frac{\ln((\rho(\theta))}{-\theta}}n^{\mu^{\prime}} $ with $  \mu^{\prime} < \mu$ and note that the sequence $ \frac{y_{n}}{x_n} $ tends to zero. Thus there exists a natural number $ n_0$ such that for all $  n \geq n_0 $, we have $ \frac{y_{n}}{x_n}< \frac{1}{2} $. Thus:\begin{eqnarray*}
\mathbb{P}(\Poisson(x_n) \leq \lfloor y_{n} \rfloor) &\leq & \mathbb{P}(\Poisson(x_n)\leq \frac{x_n}{2}) \\ &= & \mathbb{P}(\Poisson(x_n)-x_n \leq - \frac{x_n}{2}) \\ & \leq & \mathbb{P}(\vert \Poisson(x_n)-x_n \vert \geq \frac{x_n}{2}) \leq\frac{4}{x_n}
\end{eqnarray*}
by Chebyshev's inequality. By Borel-Cantelli's lemma we thus derive that for almost all $ \omega  $ there exists a natural number  $ n_1(\omega) $ such that for all $ n \geq n_1(\omega) $ we have that $ \Poisson(x_n) > \lceil y_{n-1} \rceil $. 

Now, let  $ M_{m}^{(n)} $ denote the number of boxes at generation $ n $ containing zero balls when $ m $ balls have been thrown. By Corollary \ref{zero boxes} we deduce that for almost all $ \omega $ there exists a natural number  $ n_2(\omega) $, such that for all $ n \geq n_2(\omega) $, we have $M_{y_{n-1}}^{(n)} = 0 $.
 Fix such an $ \omega $ and note that since the sequence $ (y_n) $ is increasing, for each $ m \geq y_{n_2(\omega)} $ there exists a unique $ n \geq n_2(\omega) $ such that $ y_{n-1} <  m \leq x_{n} $. Since $  \lceil y_{n-1} \rceil \leq m $ we have  $ G_m > n-1  $. Now, by taking logarithm on both sides of the inequality $ m \leq x_{n} $ and some rearrangement, we get that:
 \[n \geq \frac{-\theta}{\ln(\rho(\theta))}\ln(m)-\mu\ln(n). \]
From the inequality $ m > x_{n-1} \geq e^{(n-1)\frac{\ln(\rho(\theta))}{-\theta} }$, we get that $ \frac{\theta}{\ln(\rho(\theta))}\ln(m) \geq n $ and thus:
\begin{eqnarray}
\label{Boiler}
G_m   > n -1 \geq \frac{-\theta }{\ln(\rho(\theta))}\ln(m)-\mu\ln\left(\frac{-\theta }{\ln(\rho(\theta))}\ln(m)\right)-1, \end{eqnarray}
and we are left to show that:
\[\lim_{\substack{\theta \rightarrow -\infty }} \frac{ -\theta}{\ln(\rho(\theta))} = C_*.\]
Indeed by l'Hôpital's rule, we have that:
\[\lim_{\substack{\theta \rightarrow -\infty }}\frac{ -\theta}{\ln(\rho(\theta))}= \lim_{\substack{\theta \rightarrow -\infty }} \frac{\rho(\theta)}{-\rho^{\prime}(\theta)} =C_*, \]
and we conclude.
 \end{proof}

\begin{remark}
We could have gained the same result by applying Theorem 6.E in \cite{Barbour}.
\end{remark}

For the proof of the upper bound, we want to find an lower bound for  the number of balls $ T_n $ one needs to throw initially to observe at least one ball in each box of generation $ n $.

\begin{lemma}
  \label{markov2}

 Let $ \theta < 1 $ and define $ x_n:=e^{n\left(\frac{-\rho^{\prime}(\theta)}{\rho(\theta)}-\varepsilon^{\prime}\right)} $ , $ \varepsilon^{\prime} > 0 $. For almost all $ \omega $ there exists a natural number $ n_0(\omega) $, such that for all $ n \geq n_0(\omega) $ :
 \[\textbf{T}_n(\omega) > x_n. \]
 
 \end{lemma}
 \begin{proof}
 We use the Poissonization technique explained in the last section. We have:
 \[\mathbb{P}(\textbf{T}_n \leq x_n) = \mathbb{P}(\max_{k \leq K^n} \textbf{exp}(l_k^{(n)})  \leq x_n) \leq \mathbb{P}(\textbf{exp}(1)\leq \underline{l}^{(n)}x_n).\]
 Now, by Lemma \ref{large deviations}  there exists for each $ \varepsilon > 0  $ a natural number $ n_0(\varepsilon) $ such that $ \underline{l}^{(n)} \leq e^{n \left( \frac{\rho^{\prime}(\theta) } {\rho(\theta)} + \varepsilon \right)} $ for all $ n \geq n_0(\varepsilon) $. Let $ \varepsilon < \varepsilon^{\prime} $. We then have for each $ n \geq n_0(\varepsilon) $:
\[\mathbb{P}(\textbf{T}_n \leq x_n) \leq \mathbb{P}(\textbf{exp}(1)\leq 2e^{-n\varepsilon})=1-\exp(-e^{-n(\varepsilon^{\prime}-\varepsilon)} ) \leq e^{-n(\varepsilon^{\prime} - \varepsilon)} .\]
We conclude by Borel-Cantelli's lemma.
 \end{proof}

Since the times at which a ball is thrown are the arrival times of independent Poisson processes with rate $ 1 $, the number of balls thrown up to time $ x_n $ is Poisson distributed with parameter $ x_n $. Applying Lemma \ref{markov2}, we thus get that:
\begin{cor}
\label{zero boxes2}
 For almost all $ \omega $ it exists a natural number $ n_0(\omega) $ such that $ \forall n \geq n_0(\omega)$ there is at least one box at generation $ n $ containing zero balls when $ \Poisson(x_n) $ balls have been thrown.
\end{cor}
We now tackle the proof of the main result in this section.
\begin{proposition}
\label{Lower Bound saturation}
We have:
\[\limsup_{m \rightarrow \infty} \frac{1}{\ln(m)} G_m \leq C_*  \quad a.s.\]

\end{proposition}
\begin{proof}
Define the sequence $ y_n= e^{n \left(\frac{-\rho^{\prime}(\theta)}{\rho(\theta)}-\varepsilon \right)}  $ with $ \varepsilon^{\prime} > \varepsilon $ and notice that the sequence $ \frac{y_{n}}{x_n} $ tends to infinity. Thus there exists a natural number $ n_1 $ such that for all $ n \geq n_1 $ $ \frac{y_{n}}{x_n}> 3 $. We have:
\begin{eqnarray*}
\mathbb{P}(\Poisson(x_n) \geq \lfloor  y_{n} \rfloor) &\leq & \mathbb{P}(\Poisson(x_n) \geq 3x_n) \\ &= & \mathbb{P}(\Poisson(x_n)-x_n \geq 2x_n) \\ & \leq & \mathbb{P}(\vert \Poisson(x_n)-x_n \vert \geq 2x_n) \leq \frac{1}{4x_n},
\end{eqnarray*}
by Chebyshev's inequality. By Borel-Cantelli we thus derive that for almost all $ \omega  $ there exists a natural number  $ n_2(\omega) $ such that for all $ n \geq n_2(\omega) $ we have that $ \Poisson(x_n) < \lfloor y_{n} \rfloor$. Now by Corollary \ref{zero boxes2} we derive that for almost all $ \omega $ there exists a natural number $ n_3(\omega) $ such that for all $ n \geq n_3(\omega) $, we have:
\[M_{ \lfloor y_{n} \rfloor, j } =0. \]
 Now fix such an $ \omega $. Note that  there exists a natural number $ n_4 $ such that $ \forall n \geq n_4 $ the sequence $ y_n $ is increasing. Let $ n_5 \geq \max(n_3(\omega), n_4) $ and note that for each $ m \geq y_{n_4} $ there exists a unique $ n \geq n_4 $ such that $ y_{n} < m <y_{n+1} $. Now, since $ \lfloor y_{n-1} \rfloor \leq m  $, we have $G_m < n $. Further since $  y_n <m $, taking logarithm on both sides, we have that:
 \[\frac{1}{(\frac{-\rho^{\prime}(\theta)}{\rho(\theta)} - \varepsilon^{\prime} )}\ln(m) \geq  n > G_m.\]
 We thus derive that
 \[\limsup_{m \rightarrow \infty} \frac{1}{\ln(m)} G_m \leq \frac{1}{\frac{-\rho^{\prime}(\theta)}{\rho(\theta)} -\varepsilon},  \]
 and we conclude by letting $ \varepsilon  $ tend to zero and $ \theta $ tend to $ -\infty $.
\end{proof}
Now, let $ j \geq 2 $ and note that the first generation when  there exists a box containing no ball when $ m  $ balls have been thrown is larger then the first generation  at which there exists a box containing less than $ j $ balls when $ m  $ balls have been thrown. That is $ G_{m,j} \leq G_m $ and we derive that:
 \begin{proposition}

Let $ j \geq 1 $. We have:
\[\liminf_{m \rightarrow \infty} \frac{1}{\ln(m)} G_{m,j} \leq C_* \quad a.s. \]

\end{proposition}
\section{Occupancy scheme in random environment}
\subsection{Random probability cascades}
  Consider a random transition matrix $ A=(p_{ij}) $. To the box with label $ u $, of type $ i $, we will associate an independent copy of $A_u=(p_{ij}(u))  $ of $A $. The length of the box $ u = (i_1,...,i_{n}) $ is then given by some multiplicative cascade, that is: \begin{equation}
\label{multiplicative cascade}  l^{(n)}_u=  p_{i_1i_2}(i_1) \times ... \times p_{i_{n-1}i_{n}}((i_1,i_2,...i_{n-1})). \end{equation}
Let  $ \left(l^{(n)}_{ij,k}\right)_k $ denote the sequence  of length of the boxes of type $ j $ at generation $ n $, when the first box was of type $ i $. The process $ (Z_i^{n})_n=(Z_{i1}^n,...,Z_{iK}^n) $, with 
\[ Z_{ij}^n= \sum_k \delta_{-\ln(l_{ij,k}^{(n)})},\]
is then  a multitype branching random walk. Let $ \vert Z_{ij} \vert := \int_{\mathbb{R}} Z_{ij}(dx) $ denote the total mass. We will assume that the embedded Galton-Watson process $ ((\vert Z_{i1}^n \vert,...,\vert Z_{iK}^n\vert )$ is positive regular, that is
\begin{equation}
\label{positive regular}
\textnormal{the matrix $(\mathbb{P}( \vert Z_{ij} \vert >0) )_{ij} $ is positive regular.}
\end{equation} Let $ \theta \in \mathbb{R} $ and let us introduce the Laplace transform of the intensity measure of $ Z_{ij} $:
\[m_{ij}(\theta) =  \mathbb{E} \left[ \int e^{-\theta x} Z_{ij}(dx) \right]=\mathbb{E} \left[ p_{ij}^{\theta} \right]. \]
Let \[ L= \bigcap_{i,j} \lbrace \theta \in \mathbb{R} : m_{ij}(\theta)< \infty \rbrace \]
and note that $ L $ is an interval since $ m_{ij} $ is decreasing in $ \theta $. Let us introduce $ M(\theta)$, the matrix with entries $\mathbb{E} \left[ p_{ij}^{\theta} \right]  $, where we agree that if $ p_{ij}=0 $, then $ p_{ij}^{\theta}=0 $ even if $ \theta \leq 0 $. Note that the entries of $ M^{n}(\theta) $ are given by
 \[m_{ij}^n(\theta)=\mathbb{E} \left[ \int e^{-\theta x} Z^n_{ij}(dx) \right] =\mathbb{E}\left[ \sum_k (l_{ij,k}^{(n)})^{\theta} \right]. \] 
 Condition (\ref{positive regular}) implies that  $ M(\theta) $ is positive regular for each $ \theta \in L
 $, and thus the Perron-Frobenius theorem applies and $ M(\theta) $ possesses a maximum modulus eigenvalue $ \varrho(\theta) $, a strictly positive left-eigenvector $ w(\theta) $ and a strictly positive right-eigenvector $ v(\theta) $ with eigenvalue $ \varrho(\theta) $, such that we have $ (w(\theta))^tv(\theta)=1 $. Moreover $ \ln(\varrho(\theta)) $ is convex (see  \cite{kingman}) and analytic on $ L $ (see  \cite{biggins3}).  We will need to assume some even stronger condition on $ \varrho(\theta) $. We will assume that:
\begin{equation}
\label{strictly convex}
\textnormal{$ \ln(\varrho(\theta)) $ is strictly convex.}
\end{equation}
Similar to the function $ \psi(\theta) $ in the last section, the function
\[f(\theta):=  \ln(\varrho(\theta)) -\theta\frac{\varrho^{\prime}(\theta)}{\varrho(\theta)},\]
will play a crucial role in our analysis. Note that $ f^{\prime}(\theta)=-\theta \cdot \ln(\varrho(\theta))^{\prime\prime}$. Thus $ f $ is strictly decreasing on the interval $ (0,\infty)\cap L $ and strictly increasing on $ (-\infty,0)\cap L $. Since $ f(0)>0 $ the function $f$ is thus positive on some open interval. Let us define 
\[\theta_*=\inf\lbrace \theta \in L: f(\theta)>0 \rbrace \quad \text{and} \quad \theta^*=\sup\lbrace \theta \in L: f(\theta)>0 \rbrace. \]
The function $ f$ is then strictly positive on $ (\theta_*,\theta^*) $.

Let $( \mathcal{F}_n)$ denote the natural filtration of the  the multitype branching random walk. Following  \cite{biggins} one can define a remarkable martingale for each $ \theta \in L $, with respect to the filtration $ \mathcal{F}_n $:
 \[ W_i^n(\theta) :=\sum_{j=1}^K \frac{v_j(\theta)}{v_i(\theta)} \varrho(\theta)^{-n} \cdot \sum_k (l_{ij,k}^{(n)})^{\theta} ,\]
 where $ v_i(\theta) $ denotes the $ i $-th entry of the right-eigenvector $ v(\theta) $ with eigenvalue $ \varrho(\theta) $.
 We then have:
\begin{lemma} 
\label{Martingale convergence}
For each $ \theta \in (\theta_*,\theta^*) $ the martingale $ W_i^n(\theta) $ is bounded in $ L^{\alpha}(\mathbb{P}) $ for some $ \alpha > 1 $. It converges almost surely and in mean and its terminal value \[ W_i(\theta):= \lim_{n \rightarrow \infty} W^n_i(\theta) \] is a.s. strictly positive.

\end{lemma}
\begin{proof}
We want to apply Theorem 2 in  \cite{biggins}. We need to check that for all $ \theta \in (\theta_*,\theta^*) $ there is an $ \alpha > 1 $ such that we have $ \varrho(\alpha \theta) < \varrho(\theta)^\alpha $ and $ \max_i \mathbb{E}[(W_i^n(\theta))^{\alpha}] < \infty $. Consider the function $ g(\theta)=\frac{1}{\theta}\ln(\varrho(\theta)) $ and note that $ g^{\prime}(\theta) = -\frac{1}{\theta^2} f(\theta)$. Thus $g $ is decreasing on $(\theta_*,\theta^*)  $. For $ \alpha >1$ small enough we thus have
\[\frac{\ln(\varrho(\alpha \theta))}{\alpha \theta}<\frac{\ln(\varrho(\theta))}{\theta}\]
and thus $ \varrho(\alpha \theta)<\varrho(\theta)^{\alpha} $. For the second criterion note that
\[\mathbb{E}[(W_i^n(\theta))^{\alpha}] \leq \sum_{j=1}^K \left(\frac{v_j(\theta)}{v_i(\theta)} \varrho(\theta)^{-1}\right)^{\alpha}\mathbb{E}\left[p_{ij}^{\alpha \theta} \right] < \infty \quad \forall i,j,\]
by Jensen's inequality and the fact that $ \alpha \theta \in (\theta_*,\theta^*) $ for $ \alpha>1 $ small enough. By Theorem 2 in \cite{biggins} we thus get the convergence of $ W_i^n $ almost surely and in mean. For the a.s. strictly positivity note that 
\[\mathbb{E}[(W_i^n(\theta))^{\alpha}]=\sum_{j=1}^K \frac{v_j(\theta)}{v_i(\theta)} \varrho(\theta)^{-n} m_{ij}^n(\theta). \]
By the Perron-Frobenius theorem $\lim_{n \rightarrow \infty}  \varrho(\theta)^{-n} m_{ij}^n(\theta)= v_i(\theta)w_j(\theta)  $. Thus $ \lim_{n \rightarrow \infty} \mathbb{E}[W_i^n(\theta)] = \mathbb{E}[W_i(\theta)]=1 $ and $ \mathbb{P}(W_i(\theta)=0)=\beta_i<1 $. Moreover $ (\beta_1,...,\beta_K) $ is a fixed point of the multivariate generating function of the embedded Galton-Watson process and thus $ \beta_i $ is the extinction probability of the process started from type $ i $ and thus $ \beta_i=0 $ a.s.
\end{proof}

The following result (Corollary 2 in  \cite{biggins}),  will play the role of Corollary \ref{numerous} in the previous section. It  will help us to control the asymptotic behavior of the length of boxes at generation $ n $, as $ n $ tends to infinity. 
\begin{lemma}
\label{box}
For all $ a> b \in \mathbb{R} $ and $ \theta $ in a compact subset of $ (\theta_*,\theta^*) $ we have:
\[ \sqrt{n}e^{-n f(\theta)} \# \left\lbrace k: l_{ij,k}^{(n)} \in \left[e^{n\frac{\varrho^{\prime}(\theta)}{\varrho(\theta)}-a}, e^{n\frac{\varrho^{\prime}(\theta)}{\varrho(\theta)}-b} \right] \right\rbrace \rightarrow \frac{v_i(\theta) w_j(\theta) W_i(\theta)}{\sqrt{2 \pi f^{\prime \prime}(\theta)}} \frac{e^{a \theta}-e^{b\theta}}{\theta}\]
almost surely as $ n $ tends to infinity.
\end{lemma}
As in the last section, we will further sometimes need to control the size of the smallest and the largest box at generation $ n $, when $ n $ tends to infinity. Note that  $ \varrho(\theta) $ is decreasing in $ L $ since the entries $ m_{ij}(\theta) $ are decreasing and that the strict convexity of $ \varrho(\theta) $ thus entails that $ \varrho^\prime(\theta) < 0 $ on $ L $. Define the constants 
\[ \zeta^* :=   \lim_{\begin{subarray}{c}
\theta \rightarrow \theta^*\\
\theta \leq \theta^*
\end{subarray}} \frac{ \varrho(\theta)}{- \varrho^{\prime}(\theta)}, \quad \text{and} \quad \zeta_*:=\lim_{\substack{\theta \rightarrow \theta_* \\ \theta > \theta_* }}  \frac{ \varrho(\theta)}{- \varrho^{\prime}(\theta)}. \]
The strict convexity of $ \ln(\varrho(\theta)) $ implies that $ \frac{ \varrho(\theta)}{- \varrho^{\prime}(\theta)} $ is strictly increasing on $ L $. Recall moreover that $ \rho(\theta) $ is analytic on $ L $. For $ \theta^* < \infty $ we thus have $ 0 < \zeta^* < \infty $ since $ 0 \in L $ and  $ \frac{ \varrho(0)}{- \varrho^{\prime}(0)} >0 $.  Moreover if 
\begin{eqnarray}  \label{theta stern} -\infty < \theta_* < 0 \quad \textnormal{and}  \quad \lim_{\begin{subarray}{c}
\theta \rightarrow \theta_*\\
\theta > \theta_*
\end{subarray}} f(\theta)=0,   \end{eqnarray} then we have   $ 0 < \zeta_* < \infty $ since then $\lim_{\begin{subarray}{c}
\theta \rightarrow \theta_*\\
\theta > \theta_*
\end{subarray}} \frac{ \varrho(\theta)}{- \varrho^{\prime}(\theta)}=\lim_{\begin{subarray}{c}
\theta \rightarrow \theta_*\\
\theta > \theta_*
\end{subarray}} \frac{\ln(\varrho(\theta))}{\theta} $.

\begin{cor}
\label{largest box}
 Suppose that $ \theta^* < \infty $ and let $ \overline{l}_{i}^{(n)} $ denote the size of the  largest box  at generation $ n $, when the first box was of type $ i $. 
We then have:
\[\lim_{n \rightarrow \infty} n^{-1} \cdot \ln \left(\overline{l}_{i}^{(n)} \right)= - \frac{1}{\zeta^*} \quad a.s.\]

\end{cor}
\begin{proof}
We first show the upper bound. Let $ \theta \in (0,\theta^*) $. By Lemma \ref{Martingale convergence} we have for $ n $ large enough:
\[  \varrho(\theta)^{-n} \left(\overline{l}_{i}^{(n)}\right)^{\theta}  \leq \sum_{j=1}^K \frac{v_j(\theta)}{v_i(\theta)} \varrho(\theta)^{-n} \cdot  \left(\overline{l}_{i}^{(n)}\right)^{\theta} \leq W_i^{n}(\theta)  \leq 2W_i(\theta), \quad a.s. \]
By some rearrangement we thus get:
\[\limsup_{n \rightarrow \infty} n^{-1} \cdot \ln(\overline{l}_{i}^{(n)}) \leq \frac{\ln(\varrho(\theta))}{\theta} \quad a.s. \]
We conclude by letting $ \theta $ tend to $ \theta^* $. For the lower bound note that by Lemma \ref{box} we have  $ e^{n\frac{\varrho^{\prime}(\theta)}{\varrho(\theta)}} \leq \overline{l}_i^{(n)} $  a.s.,  and we conclude by some rearrangement and letting $ \theta  $ tend to $ \theta^* $.
\end{proof}

For the size of the smallest box one shows in the same way that:
\begin{cor}
\label{smallest box}
Let (\ref{theta stern}) hold and let $ \underline{l}_i^{(n)} $ denote the smallest box at generation $ n $, when the type of the first box was of type $ i $. We then  have
\[ \lim_{n \rightarrow \infty} n^{-1} \cdot \ln(\underline{l}_i^{(n)} ) = - \frac{1}{\zeta_*}, \quad a.s. \]
\end{cor}
\subsection{Study of the height}
Recall that $ H_{m,j}$  denotes the first generation of boxes at which all the boxes contain strictly less than $ j $ balls when $ m $  balls have been thrown independently. We will see that there is a phase transition in the limiting behavior of $ H_{m,j}$ as $ m $ tends to infinity, depending on the values of $ j $. We aim to show that:
\begin{theorem} Suppose that conditions (\ref{positive regular}) and (\ref{strictly convex}) hold. We then have:
\begin{enumerate}  \item For every $ j \in (\theta_*,\theta^*) $ we have
\[
\lim_{m \rightarrow \infty} \frac{1}{\ln(m)}  H_{m,j} = \frac{j}{-\ln(\varrho(j))}  +\textit{O}\left( \frac{\ln \ln(m)}{\ln(m)} \right)  \qquad a.s.  \]

\item  If $ \theta^* <  \infty $ we have for every $ j \geq \theta^* $ that
\[\lim_{m \rightarrow \infty} \frac{1}{\ln(m)} H_{m,j} = \zeta^* \qquad a.s. \]

\end{enumerate}
\end{theorem}
The phase transition in the limiting behavior of the height has first been observed by \cite{Joseph}.
 \begin{remark}

As in the previous sections we assume that we start from a box of type $ 1 $. We will thus drop the subscript $ 1 $ and shorthand write  $ l^{(n)}_{i,k} $ for $ l^{(n)}_{1i,k}   $ and $ \underline{l}^{(n)} $ for $ \underline{l}_1^{(n)} $ and $ W $  for $ W_1 $.
\end{remark}
We first show the upper bound.  For $ \theta \in L $ with $\theta >0  $ such that $ \lceil \theta \rceil \leq j$ and $ a>\frac{1}{\theta} $, let us define the sequence \[x_n=\varrho(\theta)^{-\frac{n}{\theta}} n^{-a}. \]
We then have:

\begin{lemma}

 For almost all $ \omega $, there exists a natural number $ n_0(\omega) $ s.t for all $ n \geq n_0(\omega) $, there exists no box at generation $ n $ containing   $ j $ or more balls when $ \lfloor x_n \rfloor $ balls have been thrown.
\end{lemma}
\begin{proof}
Let $ N^{(n)}_{m,j} $ denote the number of boxes at generation $ n $ containing   $ j $ or more balls when $ m $ balls have been thrown. Conditionally on $ \mathcal{F}_n $, the number of balls in a box at generation $ n $ of size $ l $ when $ m $ balls have been thrown has distribution $ B(m,l) $.
 Similarly to the proof of Lemma \ref{poissonboxes}, we thus have:
\begin{eqnarray*}
\mathbb{E}[ N^{(n)}_{ \lfloor x_n \rfloor,j} \vert \mathcal{F}_n] &=& \sum_{i=1}^K \sum_k \mathbb{P} \left( B \left(\lfloor x_n \rfloor, l_{i,k}^{(n)} \right) \geq j \right) \\ & \leq & \sum_{i=1}^K \sum_k  \left(  \lfloor x_n \rfloor l_{i,k}^{(n)} \right)^{\theta}
\\  & \leq & c_1(\theta)x_n^{\theta} \varrho(\theta)^n  \sum_{i=1}^K \sum_k \varrho(\theta)^{-n}\frac{v_i(\theta)}{v_1(\theta)}(l_{i,k}^{(n)})^{\theta} \\ &=& c_1(\theta)n^{-a\theta}W^{(n)}(\theta), \end{eqnarray*} where $ c_1(\theta):= \left(\max_{1 \leq i \leq K} \left( \frac{v_1(\theta)}{v_i(\theta)} \right)\right) $. We thus derive that:
\[ \mathbb{E}[N_{x_n,j}^{(n)}]  \leq c_1(\theta) n^{-a\theta} \mathbb{E} [W^{(n)}(\theta)] = c_1(\theta) n^{-a\theta}.\] We finally arrive at
\begin{eqnarray*}
\mathbb{E} \left[ \sum 1_{ \lbrace N^{(n)}_{ \lfloor x_n \rfloor,j} \geq 1 \rbrace} \right] \leq \mathbb{E} \left[\sum  N^{(n)}_{ \lfloor x_n \rfloor,j} \right] \leq  \sum c_1(\theta)  n^{-a\theta} < \infty  ,
\end{eqnarray*}
 and we conclude by Borel-Cantelli.
\end{proof}

In the same way as in the proof of Proposition \ref{upper bound}, we derive that:
\begin{proposition}

For every integer $ \theta \in L $ we have
\begin{eqnarray*}  \frac{1}{\ln(m)}  H_{m,j} \leq \frac{\theta}{-\ln(\varrho(\theta))} +\textit{O}\left( \frac{\ln \ln(m)}{\ln(m)} \right)\qquad a.s. \end{eqnarray*}
as $ m $ tends to infinity.
\end{proposition}
\begin{remark}
The function $ \theta \rightarrow   \frac{\theta}{-\ln(\varrho(\theta))}$ is decreasing on $ (\theta_*, \theta^*)$.
\end{remark}
For the lower bound let first $ j \in (\theta_*, \theta^*) $ and recall that $ H_{m,j} \geq n $, if at generation $ n $ there is at least one box containing  $ j $ or more balls.  Let  $  a> \frac{1}{2j} $ and define the sequence  \[ x_n= \varrho(\theta)^{-\frac{n}{j}}n^a. \] 
As in the previous section we then have:
\begin{lemma}

 For almost all $ \omega $ there exists a natural number $ n_0(\omega) $, such that for all $ n \geq n_0(\omega)$ there is at least one box containing $ j $ or more balls when $ \Poisson(x_n) $ balls have been thrown.
\end{lemma}
\begin{proof}
 Let \[  z_n : =  e^{n\frac{\varrho^{\prime}(j)}{\varrho(j)}}. \]
and let $ M_n $ denote the number of boxes at generation $ n $ of type $ 1 $ with size in the interval $ [z_n, 2z_n] $. Let
\[Z(j)=  \frac{v_1(j) w_1(j) W(j)}{\sqrt{2 \pi f^{\prime \prime}(j)}} \cdot \frac{1-2^{-j}}{j},\] and $ v_n(j)= \frac{1}{2}Z(j)e^{nf(j)}n^{-\frac{1}{2}} $. From Lemma \ref{box} we know that a.s. there exists a natural number $ n_0 $ such that for all $ n \geq n_0 $, we have $M_n  \geq v_n(j)  $. We can thus  a.s. consider the first $ v_n(j) $ boxes in $ M_n $, say $ b_1(n),...,b_{v_n(j)}(n) $ and denote their size with $ l_1(n),...,l_{v_n(j)}(n) $. We place an imaginary box  $ \textbf{b}_i(n) $ in $ b_i(n) $ for $ 1\leq i \leq v_n(j) $, each of size exactly $ z_n $. If a ball falls into the box $ b_i(n) $ it is placed in the imaginary box $ \textbf{b}_i(n) $ with probability $ \frac{z_n}{l_i} $.  

 Let   $ A_n $ denote the  event that the  boxes   $ \textbf{b}_i(n) $ for $ 1\leq i \leq v_n(j) $ contain strictly less then $ j $ balls when $ \Poisson(x_n) $ balls have been thrown. Since conditioned on $ \mathcal{F}_n $  the number of balls in each box of generation $ n $ are independent Poisson  random variables with parameter $ x_nz_n $, we have:
\[\mathbb{P}(A_n\vert \mathcal{F}_{\infty}) \leq \mathbb{P}(\Poisson(x_nz_n) < j)^{v_n(j)}. \]
We then finish the proof in the same way as the proof of Lemma \ref{tuttituu}.
\end{proof}

Performing the same computations as in the proof of Proposition \ref{propediprop}, we arrive at:
\begin{proposition}Suppose that (\ref{positive regular}) and (\ref{strictly convex}) hold. For every integer $ j \in (\theta_*, \theta^*) $ we have
\begin{eqnarray*}  \frac{1}{\ln(m)}  H_{m,j} \geq \frac{j}{-\ln(\varrho(j))} +\textit{O}\left( \frac{\ln \ln(m)}{\ln(m)} \right)\qquad a.s. \end{eqnarray*}
as $ m $ tends to infinity.
\end{proposition}
We now turn to the second case. Suppose that $ \theta^* < \infty $ and $ j \geq \theta^* $. Let $0<\varepsilon^{\prime} < \varepsilon$ and define the sequences
\[x_n:=e^{n\left(\frac{1}{\zeta^*}+\varepsilon  \right)} \quad \text{and} \quad y_n=e^{n\left(-\frac{1}{\zeta^*}- \varepsilon^{\prime} \right)}.\]
We then have:
\begin{lemma}
 For almost all $ \omega $ there exists a natural number $ n_0(\omega) $, s.t. $ N_{\lceil x_n \rceil,j} \geq 1 $ for all $ n \geq n_0(\omega) $.
\end{lemma}
\begin{proof}
 Let $ \overline{b}_n $ denote the largest box at generation $ n $ and recall that $ \overline{l}^{(n)} $ denotes its size. We place an imaginary box $ \textbf{b}_n $ of size $ \overline{l}^{(n)} \wedge y_n $ inside the largest box. When a ball falls into $ \overline{b}_n $ it is placed in $ \textbf{b}_n $ with probability $ \frac{y_n }{\overline{l}^{(n)}} $. Now, let \[A_n:= \lbrace  \overline{l}^{(n)} \geq y_n  \rbrace \] 
and let $ B_n $ denote the event that the box $\textbf{b}_n  $ contains strictly less than $ j $ balls when $\lceil x_n  \rceil$ balls have been thrown. As in the proof of Lemma \ref{Lemma6}, we have:
\begin{eqnarray*} \mathbb{P}(A_n \cap B_n) &\leq& \mathbb{P}(B_n \vert A_n)  \\ &\leq& \mathbb{P}((B(\lceil x_n \rceil, y_n )+1)^{-1} \leq j^{-1}) \\& \leq & \frac{j}{ x_n  \cdot y_n  }, \end{eqnarray*}

and we arrive at 
\[\mathbb{P}(A_n \cap B_n) \leq j e^{-n (\varepsilon - \varepsilon^{\prime})}. \]
We conclude by Borel-Canteli's lemma and the fact that $ \mathbb{P}(A_n)=1 $ by Corollary \ref{largest box}.
\end{proof}

Performing the usual computations, we arrive at:

\begin{proposition} Suppose that (\ref{positive regular}),  (\ref{strictly convex}) hold, that $ \theta^ * < \infty $ and $ j\geq \theta^* $. We then have
\[\liminf_{m \rightarrow \infty} \frac{1}{\ln(m)} H_{m,j} \geq \zeta^* \qquad a.s. \]

\end{proposition}
\subsection{Study of the saturation level}
 We aim to show that:
\begin{theorem} Suppose that (\ref{positive regular}), (\ref{strictly convex}) and (\ref{theta stern}) hold. We then have
\[\lim_{m \rightarrow \infty} \frac{1}{\ln(m)} G_{m,j} = \zeta_* \quad a.s. \] 
\end{theorem}
We first tackle the proof of the lower bound. Let $ \varepsilon^{\prime} > 0 $ and define the sequence 
\[x_n=e^{n(\frac{1}{\zeta_*} + \varepsilon^{\prime})}, \]
where, $ \zeta_* $  is the constant appearing in the limit behavior of the smallest box. In the same notation as in the previous sections we then have:
\begin{lemma}
 For almost all $ \omega $ there exists a natural number $ n_0(\omega) $, such that for all $ n \geq n_0(\omega) $ :
 \[\textbf{T}_n(\omega) < x_n. \]
 \end{lemma}
 \begin{proof}
Let $ 0 < \varepsilon < \varepsilon^{\prime} $. Define $ A_n$, the event that the smallest box  $\underline{l}^{(n)}  $ of generation $ n $ is larger than $ e^{n(-\frac{1}{\zeta_*}-  \varepsilon)} $ and $ B_n $, the event that $ \textbf{T}_n \geq x_n $. Let $ \textbf{exp}_k(1) $ for $ k \leq K^n$ denote independent exponential random variables with parameter $ 1 $. We then have:
\begin{eqnarray*}
\mathbb{P}(B_n \cap A_n) & \leq & \mathbb{P}(B_n \vert A_n) \\  & \leq & \mathbb{P}(\max_{k \leq K^n} \textbf{exp}_k(1) \geq e^{n(-\frac{1}{\zeta_*} - \varepsilon )}x_n) \\ & \leq & 2 \cdot \mathbb{P}(\textbf{G}(1) \geq e^{n(-\frac{1}{\zeta_*} - \varepsilon )}x_n-n\ln(K)) \\ & \leq & 2 \cdot \exp(-(e^{n(  \varepsilon^{\prime} - \varepsilon)}  -n\ln(K)),
\end{eqnarray*}
and we conclude by the fact that $ \mathbb{P}(A_n)=1 $ by Corollary \ref{smallest box} and Borel-Cantelli's lemma.
 \end{proof}

 Performing the same calculations and generalizations as in the previous section, we arrive at:
 \begin{proposition} Suppose that (\ref{positive regular}), (\ref{strictly convex}) and (\ref{theta stern}) hold. We have:
 \[\liminf_{m \rightarrow \infty} \frac{1}{\ln(m)} G_{m,j} \geq \zeta_* \quad a.s.  \]

 \end{proposition}
 For the upper bound let $ \varepsilon^{\prime} >0 $ and define the sequence \[ y_n= e^{n(\frac{1}{\zeta_*} - \varepsilon^{\prime})}.\]
 We then have:
 \begin{lemma}

 For almost all $ \omega $ there exists a natural number $ n_0(\omega) $, such that for all $ n \geq n_0(\omega) $ :
 \[\textbf{T}_n(\omega) > y_n. \]
 
 \end{lemma}
 \begin{proof}
 Let $ 0 < \varepsilon^{\prime} < \varepsilon $. Define $ A_n$, the event that the smallest box  $\underline{l}^{(n)}  $ of generation $ n $ is smaller than $ e^{n(-\frac{1}{\zeta_*} + \varepsilon)} $ and $ B_n $, the event that $ \textbf{T}_n \leq y_n $.  We then have:
\begin{eqnarray*}
\mathbb{P}(B_n \cap A_n) & \leq &\mathbb{P}(B_n \vert A_n)   \\  & \leq & \mathbb{P}( \textbf{exp}_k(1) \leq e^{n(\frac{1}{\zeta_*} + \varepsilon )}y_n) \\ & \leq & 1-\exp(-e^{-n(\varepsilon-\varepsilon^{\prime})} ) \\ & \leq & e^{-n(\varepsilon - \varepsilon^{\prime})} ,
\end{eqnarray*}
and we conclude by the fact that $ \mathbb{P}(A_n)=1 $ by Corollary \ref{smallest box} and Borel-Cantelli's lemma.
 \end{proof}

 The upper bound then easily follows.
 
\acknowledgements
 I would like to thank Jean Bertoin for introducing me to this problem and for his advice and support. I would also like to thank two anonymous referees  for carefully reading the first draft of this work and for their helpful comments.
\nocite{*}
\bibliographystyle{abbrvnat}
\bibliography{sample-dmtcs}

\begin{thebibliography}{25}
\providecommand{\natexlab}[1]{#1}
\providecommand{\url}[1]{\texttt{#1}}
\expandafter\ifx\csname urlstyle\endcsname\relax
  \providecommand{\doi}[1]{doi: #1}\else
  \providecommand{\doi}{doi: \begingroup \urlstyle{rm}\Url}\fi

\bibitem[Ash(2012)]{Ash}
R.~B. Ash.
\newblock \emph{{Information Theory}}.
\newblock Dover Publications, New York, 2012.

\bibitem[Barbour et~al.(1992)Barbour, Holst, and Janson]{Barbour}
A.~D. Barbour, L.~Holst, and S.~Janson.
\newblock \emph{{Poisson Approximation}}.
\newblock Oxford University Press, Oxford, 1992.

\bibitem[Bertoin(2008)]{Jean}
J.~Bertoin.
\newblock {Asymptotic regimes for the occupancy scheme of multiplicative
  cascades}.
\newblock \emph{Stochastic Processes and their Applications}, 118:\penalty0
  1586--1605, 2008.

\bibitem[Biggins(1992)]{biggins2}
J.~D. Biggins.
\newblock {Uniform convergence of martingales in the branching random walk}.
\newblock \emph{Annals of Probability}, 20\penalty0 (1):\penalty0 137--151,
  1992.

\bibitem[Biggins and Rahimzadeh~Sani(2004)]{biggins3}
J.~D. Biggins and A.~Rahimzadeh~Sani.
\newblock {Extended Perron-Frobenius results}.
\newblock 2004.
\newblock available at
  \href{http://biggins.staff.shef.ac.uk/epfr.pdf}{\texttt{http://biggins.staff.shef.ac.uk/epfr.pdf}}.

\bibitem[Biggins and Rahimzadeh~Sani(2005)]{biggins}
J.~D. Biggins and A.~Rahimzadeh~Sani.
\newblock {Convergence results on multitype multivariate branching random
  walks}.
\newblock \emph{Advances in Applied Probability}, 37\penalty0 (3):\penalty0
  681–705, 2005.

\bibitem[Dembo and Zeitouni(2010)]{ofer}
A.~Dembo and O.~Zeitouni.
\newblock \emph{{Large Deviations Techniques and Applications}}.
\newblock Springer, New York, 2010.

\bibitem[Devroye(1984)]{Devroye}
L.~Devroye.
\newblock { A probabilistic analysis of the height of tries and of the
  complexity of triesort}.
\newblock \emph{Acta Informatica}, 21:\penalty0 229–237, 1984.

\bibitem[Devroye(2005)]{Devroye2}
L.~Devroye.
\newblock {Universal asymptotics for random tries and patricia trees}.
\newblock \emph{Algorithmica}, 42:\penalty0 11–29, 2005.

\bibitem[Drmota(2009)]{drmota}
M.~Drmota.
\newblock \emph{{Random Trees}}.
\newblock Springer, Wien, New York, 2009.

\bibitem[Flajolet and Steyaert(1982)]{Flajolet}
P.~Flajolet and J.-M. Steyaert.
\newblock { A branching process arising in dynamic hashing, trie searching and
  polynomial factorization}.
\newblock In \emph{Lecture Notes in Computer Science}, volume 140. 1982.

\bibitem[Gnedin(2004)]{Gnedin}
A.~Gnedin.
\newblock {The Bernoulli sieve}.
\newblock \emph{Bernoulli}, 10\penalty0 (1):\penalty0 79--96, 2004.

\bibitem[Gnedin et~al.(2007)Gnedin, Hansen, and Pitman]{poissonisation}
A.~Gnedin, B.~Hansen, and J.~Pitman.
\newblock {Notes on the occupancy problem with infinitely many boxes: general
  asymptotics and power laws}.
\newblock \emph{Probability Surveys}, \penalty0 (4):\penalty0 146--171, 2007.

\bibitem[Gnedin et~al.(2010)Gnedin, Iksanov, and Marynych]{Gnedin2}
A.~Gnedin, A.~Iksanov, and A.~Marynych.
\newblock {The Bernoulli sieve: an overview}.
\newblock \emph{In Proceedings of the 21st International Meeting on
  Probabilistic, Combinatorial, and Asymptotic Methods in the Analysis of
  Algorithms (AofA10), Discrete Math. Theor. comput. Sci. AM}, pages 329--341,
  2010.

\bibitem[Holst(1986)]{Holst}
L.~Holst.
\newblock {On Birthday, Collector's, Occupancy and other Classical Urn
  Problems}.
\newblock \emph{International Statistical Review}, 54\penalty0 (1):\penalty0
  15--27, 1986.

\bibitem[Johnson and Kotz(1977)]{Kotz}
N.~L. Johnson and S.~Kotz.
\newblock \emph{{Urn models and their application}}.
\newblock John Wiley and Sons, New York, London, Sydney, 1977.

\bibitem[Joseph(2010)]{Joseph}
A.~Joseph.
\newblock {A Phase Transition for the Heights of a Fragmentation Tree}.
\newblock \emph{Random Structures and Algorithms}, 39:\penalty0 247--274, 2010.

\bibitem[Kingman(1961)]{kingman}
J.~F.~C. Kingman.
\newblock {A convexity property of positive matrices}.
\newblock \emph{The Quarterly Journal of Mathematics}, 12\penalty0
  (1):\penalty0 283--284, 1961.

\bibitem[Kolchin et~al.(1978)Kolchin, Sevast’yanov, and Chistyakov]{Kolchin}
V.~F. Kolchin, B.~A. Sevast’yanov, and V.~P. Chistyakov.
\newblock \emph{{Random Allocations}}.
\newblock John Wiley and Sons, New York, London, Sydney, 1978.

\bibitem[Pittel(1985)]{Pittel}
B.~Pittel.
\newblock {Asymptotical growth of a class of random trees}.
\newblock \emph{Annals of Probability}, 13\penalty0 (2):\penalty0 414--427,
  1985.

\bibitem[Pittel(1986)]{Pittel2}
B.~Pittel.
\newblock {Path in a Random Digital Tree: Limiting Distributions}.
\newblock \emph{Advances in Applied Probability}, 18\penalty0 (1):\penalty0
  139--155, 1986.

\bibitem[Robert and Simatos(2009)]{Robert}
P.~Robert and F.~Simatos.
\newblock {Occupancy scheme associated to Yule processes}.
\newblock \emph{Advances in Applied Probability}, 41\penalty0 (2):\penalty0
  600--622, 2009.

\bibitem[Rosen(1970)]{coupon}
B.~Rosen.
\newblock {On the coupon collector's waiting time}.
\newblock \emph{The Annals of Mathematical Statistics}, 41\penalty0
  (6):\penalty0 1952--1969, 1970.

\bibitem[Szpankowski(1991)]{S}
W.~Szpankowski.
\newblock {On the Height of Digital Trees and related Problems}.
\newblock \emph{Algorithmica}, 6:\penalty0 256--277, 1991.

\bibitem[Szpankowski(2011)]{S2}
W.~Szpankowski.
\newblock \emph{{Average case analysis of algorithms on sequences}}.
\newblock John Wiley and Sons, New York, 2011.

\end{thebibliography}
\label{sec:biblio}

\end{document}